\newif\ifsisc
\ifsisc
\documentclass[a4paper,11pt,reqno]{amsart}
\else
\documentclass[a4paper,12pt,reqno]{amsart}
\fi
\usepackage{moreverb}
\usepackage[T1]{fontenc}
\usepackage[utf8]{inputenc}
\usepackage[english]{babel}
\usepackage[left=25mm,right=25mm,top=35mm,bottom=35mm]{geometry}
\usepackage{amsmath,amssymb,amsthm}
\usepackage{pgfplotstable}
\usepackage[exponent-product = \cdot]{siunitx}
\pgfplotsset{
tick label style = {font = \tiny},
legend style = {font = \tiny},
xlabel style={yshift=+0.5ex},
ylabel style={yshift=-1.0ex}
}
\usepackage{subcaption}
\usepackage[margin=0pt]{caption}
\usepackage[hidelinks]{hyperref}
\newcommand\0{\boldsymbol{0}}
\newcommand{\eps}{\varepsilon}
\renewcommand{\d}{\mathrm{d}}

\newcommand{\y}{\mathbf{y}}

\newcommand{\N}{\mathbb{N}}
\renewcommand{\P}{\mathbb{P}}
\newcommand{\R}{\mathbb{R}}

\newcommand{\V}{\boldsymbol{\mathbb{V}}}
\newcommand{\X}{\mathbb{X}}
\newcommand{\GG}{\mathcal{G}}

\newcommand{\MM}{\mathcal{M}}
\newcommand{\NN}{\mathcal{N}}

\renewcommand{\SS}{\mathcal{S}}
\newcommand{\TT}{\mathcal{T}}

\newcommand\III{\mathfrak{I}}
\newcommand\MMM{\mathfrak{M}}
\newcommand\PPP{\mathfrak{P}}
\newcommand\QQQ{\mathfrak{Q}}
\newcommand{\qsat}{q_{\mathrm{sat}}}

\DeclareMathOperator*{\hull}{span}
\DeclareMathOperator*{\refine}{refine}
\DeclareMathOperator*{\supp}{supp}
\newcommand{\enorm}[3][]{#1|\!#1|\!#1|\,#2\,#1|\!#1|\!#1|_{#3}}
\newcommand{\bnorm}[2][]{#1|\!#1|\!#1|\,#2\,#1|\!#1|\!#1|}

\newcommand{\norm}[3][]{#1\|#2#1\|_{#3}}

\def\aa{\boldsymbol{a}}
\def\ee{\boldsymbol{e}}
\def\ff{\boldsymbol{f}}
\def\uu{\boldsymbol{u}}
\def\vv{\boldsymbol{v}}
\def\ww{\boldsymbol{w}}
\def\zz{\boldsymbol{z}}

\newcommand{\dual}[3][]{#1\langle#2\,,\,#3#1\rangle_{D}}
\newcommand{\Dual}[3][]{#1\langle#2\,,\,#3#1\rangle_{D \times \Gamma}}

\def\reff#1#2{\stackrel{\eqref{#1}}{#2}}
\def\refp#1#2{\stackrel{\phantom{\eqref{#1}}}{#2}}
\def\preff#1#2{\stackrel{\phantom{\eqref{#1}}}{#2}}
\def\set#1#2{\big\{#1 \,:\, #2\big\}}
\def\coarse{\bullet}
\def\fine{\circ}

\def\Refine{{\mbox{\scriptsize$\pmb{\mathbb{REFINE}}$}}}
\def\P{\pmb{\mathbb{P}}}
\def\M{\pmb{\mathbb{M}}}
\def\Cest{C_{\rm est}}

\def\Cgoal{C_{\rm goal}}
\def\zz{\boldsymbol{z}}

\def\Crel{C_{\rm rel}}

\def\gbold{\boldsymbol{g}}

\def\dx{\,dx}
\def\dpi{\,d\pi}

\newcommand\tol{\mathsf{tol}}

\newcommand\rev[1]{{\color{black}#1}}%{#1}%
\newcommand\revx[1]{{\color{black}#1}}%{#1}%

\newtheorem{theorem}{Theorem}
\newtheorem{proposition}[theorem]{Proposition}
\newtheorem{lemma}[theorem]{Lemma}
\newtheorem{corollary}[theorem]{Corollary}
\newtheorem{algorithm}[theorem]{Algorithm}
\newtheorem{remark}[theorem]{Remark}

\usepackage{fancyhdr}
\advance\footskip0.5cm
\title{Goal-oriented adaptivity for multilevel stochastic Galerkin FEM with nonlinear goal functionals}
\author{Alex Bespalov}
\address{School of Mathematics, University of Birmingham, Edgbaston, Birmingham B15 2TT, United Kingdom}
\email{a.bespalov@bham.ac.uk}
\author{Dirk Praetorius}
\address{Institute of Analysis and Scientific Computing, TU Wien, Wiedner Hauptstra\ss{}e~8--10, 1040 Vienna, Austria}
\email{dirk.praetorius@asc.tuwien.ac.at}
\author{Michele Ruggeri}
\address{Department of Mathematics, University of Bologna, Piazza di Porta San Donato 5, 40126 Bologna, Italy}
\email{m.ruggeri@unibo.it}
\date{\today}
\subjclass[2010]{35R60, 65C20, 65N30, 65N12, 65N15, 65N50}
\keywords{goal-oriented adaptivity, nonlinear goal functionals, a~posteriori error analysis,
multilevel stochastic Galerkin method, finite element method, parametric PDEs}
\thanks{\emph{Acknowledgments.}
The work of the first author was supported by the EPSRC under grant EP/P013791/1.
The work of the second author was supported by the Austrian Science Fund (FWF) under grants F65 and~P33216.
The third author is a member of the `Gruppo Nazionale per il Calcolo Scientifico (GNCS)'
of the Italian `Istituto Nazionale di Alta Matematica (INdAM)'
and was partially supported by the European Union - NextGenerationEU under the National Recovery and Resilience Plan (PNRR) - Mission 4 Education and research - Component 2 From research to business - Investment 1.1 Notice Prin 2022 - DD N. 104 of 2/2/2022, entitled \emph{Low-rank Structures and Numerical Methods in Matrix and Tensor Computations and their Application}, code 20227PCCKZ -- CUP J53D23003620006.
All authors would like to thank the Erwin Schr\"odinger International Institute for Mathematics and Physics (ESI)
at the University of Vienna for support and hospitality during the workshops on
\emph{Adaptivity, high dimensionality and randomness} (April 4--8, 2022) and
\emph{Approximation of high-dimensional parametric PDEs in forward UQ} (May 9--13, 2022),
where part of the work on this paper was undertaken.}
\begin{document}

%%%%%%%%%%%%%%%%%%%%
\begin{abstract}
This paper is concerned with the numerical approximation of quantities of interest associated with solutions
to parametric elliptic partial differential equations (PDEs).
The key novelty of this work is in its focus on the quantities of interest represented
by continuously G{\^ a}teaux differentiable \emph{nonlinear} functionals.
We consider a class of parametric elliptic PDEs where
the underlying differential operator has affine dependence on a countably infinite number of uncertain parameters.
We design a goal-oriented adaptive algorithm for approximating nonlinear functionals of solutions to this class of parametric PDEs.
In the algorithm, the approximations of parametric solutions to the primal and dual problems are computed
using the \emph{multilevel} stochastic Galerkin finite element method (SGFEM) and
the adaptive refinement process is guided by reliable spatial and parametric error reduction indicators.
We prove that the proposed algorithm generates multilevel SGFEM approximations for which
the estimates of the error in the goal functional converge to zero.
Numerical experiments for a selection of test problems and nonlinear quantities of interest illustrate and underpin our theoretical findings.
\end{abstract}
%%%%%%%%%%%%%%%%%%%%

%%%%%%%%%%%%%%%%%%%%
\maketitle
\thispagestyle{fancy}
%%%%%%%%%%%%%%%%%%%%

%%%%%%%%%%%%%%%%%%%%

%%%%%%%%%%%%%%%%%%%%
\section{Introduction} \label{sec:intro}
%%%%%%%%%%%%%%%%%%%%

Numerical approximation methods and efficient solution strategies for high-dimensional parametric partial differential equations (PDEs)
have received a significant attention in the last two decades, particularly in the context of uncertainty quantification;
see the review articles~\cite{sg11, gwz14, CohenDeVore15}.
The focus of the work presented in this paper is on the design and analysis of
adaptive algorithms that generate accurate approximations of, in general, \emph{nonlinear} quantities of interest (QoIs)
derived from solutions to parametric elliptic PDEs.
While the classical Monte Carlo sampling and its more efficient modern variants (such as Quasi-Monte Carlo and multilevel Monte Carlo)
are effective in estimating the moments of solutions, the surrogate approximations that are \emph{functions} of the stochastic parameters
can be used to estimate a wide range of QoIs derived from solutions.
Two variants of surrogate approximations, both based on spatial discretizations with the finite element method (FEM),
have been extensively studied:
\emph{stochastic collocation FEMs} generate uncoupled discrete problems by sampling the PDE inputs at deterministically chosen points
(typically, the nodes of a sparse grid) and build a multivariate interpolant from the sampled discrete solutions;
in \emph{stochastic Galerkin FEMs}, the approximations are defined via Galerkin projection and represented as finite (sparse)
generalized polynomial chaos (gPC) expansions whose spatial coefficients are computed by solving a single fully coupled discrete system.

Numerical approximations of QoIs derived from solutions to parametric PDEs have been addressed in a number of works.
The multilevel Monte Carlo (MLMC) algorithm for estimating bounded linear functionals and
continuously Fr{\' e}chet differentiable nonlinear functionals of the solution has been studied
in~\cite{CharrierST_13_FEE} and~\cite{tsgu2013} for a large class of elliptic PDEs with random coefficients.
In particular, the convergence with optimal rates for MLMC approximations of nonlinear output functionals has been proved
in~\cite{tsgu2013} using the duality technique from~\cite{gilesSuli2002}.
In the same context of using the MLMC for estimating QoIs, an \emph{adaptive algorithm}
based on goal-oriented \textsl{a posteriori} error estimation has been developed in~\cite{emn16}.
The proposed algorithm performs a problem-dependent adaptive refinement of the MLMC mesh hierarchy
aiming to control the error in the QoI and thus substantially reducing the complexity of MLMC computations.

Goal-oriented \textsl{a~posteriori} error estimates 
for generic \emph{surrogate approximations} of solutions to parametric PDEs have been proposed in~\cite{bryantPrudhommeWildey2015}
and specifically for \emph{stochastic collocation} approximations
in~\rev{\cite{almeidaOden2010, bprs25}}.
\rev{In these} works, various goal-oriented adaptive refinement strategies guided by the error estimates are discussed and
tested for model PDE problems with inputs that depend on a \emph{finite} number of uncertain parameters.
In the context of \emph{stochastic Galerkin FEM} (SGFEM), the \textsl{a~posteriori} error estimation of linear functionals of solutions
was addressed in~\cite{mathelinLeMaitre2007, bprr18+} and, for nonlinear problems, in~\cite{butlerDawsonWildey2011}.
In particular, in our previous work~\cite{bprr18+}, we considered a class of parametric elliptic PDEs where the underlying differential operator
had affine dependence on a \emph{countably infinite} number of uncertain parameters.
We used the duality technique (e.g., from~\cite{gilesSuli2002}) to design a goal-oriented adaptive SGFEM algorithm
for accurate approximation of moments of linear functionals of the solution to this class of PDE problems.
In the algorithm, the solutions to the primal and dual problems were computed using the SGFEM in its 
simplest (albeit converging with suboptimal rates) %(so-called
\emph{single-level} variant,
where all spatial coefficients in the gPC expansion resided in the same finite element~space.

In this paper, we extend the 
results of~\cite{bprr18+} in three directions.
Firstly, we extend the goal-oriented \textsl{a posteriori} error analysis in~\cite{bprr18+} and the associated adaptive algorithm
to a class of continuously G{\^ a}teaux differentiable \emph{nonlinear} goal functionals.
Secondly, aiming for optimal convergence rates, we employ the \emph{multilevel} variant of SGFEM,
where different spatial gPC-coefficients are allowed to reside in different finite element spaces;
see~\cite{egsz14, cpb18+, bpr2020+, bpr2021+, BachmayrEEV_CAF}.
Finally, we prove the convergence result for the proposed goal-oriented adaptive algorithm
(thus, providing a theoretical guarantee that, given any positive error tolerance, the algorithm stops after a finite number of iterations).
We also demonstrate in a series of numerical experiments that
for certain parametric problems and for some classes of nonlinear goal functionals,
the proposed goal-oriented adaptive strategy yields optimal convergence rates
(for both the error estimates and the reference errors in nonlinear quantities of interest)
with respect to the overall dimension of the underlying multilevel approximations spaces.

The paper is organized as follows.
Section~\ref{sec:problem} introduces the parametric PDE problem that we consider in this work along with its weak formulation.
In section~\ref{sec:mlsgfem}, we follow~\cite{bpr2020+, bpr2021+} and recall the main ingredients of the multilevel SGFEM
as well as the computable
energy error estimates for multilevel SGFEM approximations.
Focusing on a class of nonlinear goal functionals,
section~\ref{sec:goafem_nonlinear} addresses the goal-oriented error estimation as well as
the design of the goal-oriented adaptive algorithm and its convergence analysis.
The results of numerical experiments are reported in section~\ref{sec:numerics}.

%%%%%%%%%%%%%%%%%%%%
\section{Problem formulation}  \label{sec:problem}
%%%%%%%%%%%%%%%%%%%%

Let $D \subset \R^d$, $d \in \{2,3\}$, be a bounded Lipschitz domain with polytopal boundary~$\partial D$,
endowed with the standard Lebesgue measure.
With $\Gamma := \prod_{m=1}^\infty [-1,1]$ denoting the infinitely-dimensional hypercube,
we consider a probability space $(\Gamma,\mathcal{B}(\Gamma),\pi)$.
Here, $\mathcal{B}(\Gamma)$ is the Borel $\sigma$-algebra on $\Gamma$
and $\pi$ is a probability measure,
which we assume to be the product of symmetric Borel probability measures $\pi_m$ on~$[-1,1]$,
i.e., $\pi(\y) = \prod_{m=1}^\infty \pi_m(y_m)$ for all $\y = (y_m)_{m \in \N} \in \Gamma$.
We refer to $D$ and $\Gamma$ as the \emph{physical domain}
and the \emph{parameter domain}, respectively.

We aim to approximate a functional value $\gbold(\uu) \in \R$,
where $\uu \colon D \times \Gamma \to \R$ solves the stationary diffusion problem
\begin{equation} \label{eq:strongform}
\begin{aligned}
 -\nabla \cdot (\aa(x,\y) \nabla \uu(x,\y) ) &= \ff(x,\y), \quad && x \in D, \, \y \in \Gamma,\\
\uu(x,\y) & = 0, \quad && x \in \partial D, \, \y \in \Gamma.
\end{aligned}
\end{equation}
In~\eqref{eq:strongform},
the differential operators are taken with respect to the spatial variable $x \in D$.
We assume that $\ff \,{\in}\, {L^2_\pi(\Gamma;H^{-1}(D))}$
and that the diffusion coefficient $\aa$ has affine dependence on the parameters,
i.e., there holds
\begin{equation} \label{eq1:a}
\aa(x,\y) = a_0(x) + \sum_{m=1}^\infty y_m a_m(x)
\quad \text{for all } x \in D \text{ and } \y = (y_m)_{m\in\N} \in \Gamma.
\end{equation}
We suppose that the scalar functions $a_m \in L^{\infty}(D)$ in~\eqref{eq1:a}
satisfy the following inequalities (cf.~\cite[section~2.3]{sg11}):
\begin{gather}
\label{eq2:a}
 0 < a_0^{\rm min} \le a_0(x) \le a_0^{\rm max} < \infty
 \quad \text{for almost all } x \in D,\\
\label{eq3:a}
 \tau := \frac{1}{a_0^{\rm min}} \, \bigg\| \sum_{m=1}^\infty |a_m| \bigg\|_{L^\infty(D)} < 1
 \text{\ \ and\ \ }
 \sum_{m=1}^\infty \norm{a_m}{L^\infty(D)} < \infty.
\end{gather}

Let $\X := H^1_0(D)$.
We consider the Bochner space $\V := L^2_\pi(\Gamma;\X)$
and define the following symmetric bilinear forms on $\V$:
\begin{align}\label{def:B0}
 B_0(\uu,\vv) &:= \int_\Gamma \int_D a_0(x) \nabla \uu(x,\y) \cdot \nabla \vv(x,\y) \,\d{x} \,\d{\pi(\y)},
 \\\label{def:B}
 B(\uu,\vv) &:= B_0(\uu,\vv)
 + \sum_{m=1}^\infty \int_\Gamma \int_D y_ma_m(x) \nabla \uu(x,\y) \cdot \nabla \vv(x,\y) \, \d{x} \, \d{\pi(\y)}.
\end{align}
Owing to~\eqref{eq1:a}--\eqref{eq3:a},
the bilinear forms $B(\cdot,\cdot)$ and $B_0(\cdot,\cdot)$ are continuous and elliptic on $\V$.
Moreover,
the norms they induce on $\V$, denoted by $\bnorm{\cdot}$ and $\enorm{\cdot}{0}$, respectively,
are equivalent in the sense that
\begin{equation}\label{eq:lambda}
 \lambda \, \enorm{\vv}{0}^2 \le \bnorm{\vv}^2 \le \Lambda \, \enorm{\vv}{0}^2
 \quad \text{for all } \vv \in \V,
\end{equation}
where the constants
$\lambda := 1-\tau$ and $\Lambda := 1+\tau$ satisfy $0 < \lambda < 1 < \Lambda < 2$.

The weak formulation of~\eqref{eq:strongform} reads as follows:
Find $\uu \in \V$ such that
\begin{equation} \label{eq:weakform}
B(\uu,\vv) = F(\vv) := \int_\Gamma \int_D \ff(x,\y) \vv(x,\y) \, \d{x} \, \d{\pi(\y)}
\quad \text{for all } \vv \in \V.
\end{equation}
The existence of a unique solution $\uu \in \V$ to~\eqref{eq:weakform} is guaranteed by the Riesz theorem.
Throughout this work, we will refer to~\eqref{eq:weakform} as the \emph{primal problem}.

Since we aim to approximate $\gbold(\uu) \approx \gbold(\uu_\coarse)$
by the functional value attained by an approximation $\uu_\coarse \approx \uu$,
we assume that the goal functional $\gbold \colon \V \to \R$ is continuous\footnote{Here and throughout the paper,
we use $\bullet$ as a placeholder for the iteration counter;
see, e.g., $\uu_\ell$ in Algorithm~\ref{algorithm}.
The notation is identical to that used in~\cite{bpr2020+, bpr2021+}.}.
Further assumptions on $\gbold$ will be specified later.

%%%%%%%%%%%%%%%%%%%%
\section{Multilevel SGFEM discretization} \label{sec:mlsgfem}
%%%%%%%%%%%%%%%%%%%%

In this section,
we introduce the main ingredients of the multilevel SGFEM discretization
employed in our goal-oriented adaptive algorithms.
We follow the approach (and the notation) of~\cite{bpr2020+,bpr2021+}.

%%%%%%%%%%%%%%%%%%%%
\subsection{Discretization in the physical domain and mesh refinement}
\label{sec:spatial_discretization}
%%%%%%%%%%%%%%%%%%%%

Let $\TT_\coarse$ be a \emph{mesh}, i.e.,
a regular finite partition of $D \subset \R^d$
into compact nondegenerate simplices (i.e., triangles for $d=2$ and tetrahedra for $d=3$).
Let $\NN_\coarse$ denote the set of vertices of $\TT_\coarse$. 
For mesh refinement,
we employ newest vertex bisection (NVB)~\cite{stevenson}.
We consider a (coarse) initial mesh $\TT_0$
and
denote by $\refine(\TT_0)$ the set of all meshes obtained from $\TT_0$
by performing finitely many steps of NVB refinement.
Throughout this work,
we assume that all meshes used for the discretization in the physical domain
belong to $\refine(\TT_0)$.

For each mesh $\TT_\coarse \in \refine(\TT_0)$,
we denote by $\widehat\TT_\coarse$ its uniform refinement.
For $d=2$,
$\widehat\TT_\coarse$
is the mesh obtained by decomposing each element of $\TT_\coarse$
into four triangles using three successive bisections.
For $d=3$, 
we refer to~\cite[Figure~3]{egp18+} and the associated discussion therein.
Let $\widehat\NN_\coarse$ be the set of vertices of $\widehat\TT_\coarse$.
We denote by $\NN_\coarse^+ := (\widehat\NN_\coarse \setminus \NN_\coarse) \setminus \partial D$
the set of new interior vertices
created by uniform refinement of $\TT_\coarse$.
For a set of marked vertices $\MM_\coarse \subseteq \NN_\coarse^+$,
let $\TT_\fine := \refine(\TT_\coarse,\MM_\coarse)$ be the coarsest mesh such that
$\MM_\coarse \subseteq \NN_\fine$,
i.e., all marked vertices of $\TT_\coarse$ are vertices of $\TT_\fine$.
Since NVB is a binary refinement rule, 
it follows that $\NN_\fine \subseteq \widehat\NN_\coarse$
and $(\NN_\fine \setminus \NN_\coarse) \setminus \partial D = \NN_\coarse^+ \cap \NN_\fine$.
In particular, the choices $\MM_\coarse = \emptyset$ and $\MM_\coarse = \NN_\coarse^+$ lead to the meshes
$\TT_\coarse = \refine(\TT_\coarse,\emptyset)$
and $\widehat\TT_\coarse = \refine(\TT_\coarse,\NN_\coarse^+)$, respectively.

With each mesh $\TT_\coarse \in \refine(\TT_0)$,
we associate the finite element space
\begin{equation*}
 \X_\coarse := \SS^1_0(\TT_\coarse)
 := \{v_\coarse \in \X : v_\coarse \vert_T \text{ is affine for all } T \in \TT_\coarse \} \subset \X = H^1_0(D),
\end{equation*}
consisting of globally continuous and $\TT_\coarse$-piecewise affine functions.
We denote by
$\{ \varphi_{\coarse,\xi} : \xi \in \NN_\coarse \setminus \partial D \}$ the basis of $\X_\coarse$
comprising the so-called hat functions,
i.e.,
for all $\xi \in \NN_\coarse$, $\varphi_{\coarse,\xi} \in \X_\coarse$
satisfies the Kronecker property $\varphi_{\coarse,\xi}(\xi') = \delta_{\xi\xi'}$ for all $\xi' \in \NN_\coarse$.
Consistent with this notation, $\widehat\X_\coarse := \SS^1_0(\widehat\TT_\coarse)$ denotes the finite element space
associated with the uniform refinement $\widehat\TT_\coarse$ of $\TT_\coarse$, and
$\{ \widehat\varphi_{\coarse,\xi} : \xi \in \widehat\NN_\coarse \setminus \partial D \}$ is
the corresponding set of hat functions (the basis of~$\widehat\X_\coarse$).
There holds the ($H^1$-stable) two-level decomposition
$\widehat\X_\coarse = \X_\coarse \oplus \hull\{ \widehat\varphi_{\coarse,\xi} : \xi \in \NN_\coarse^+ \}$.

%%%%%%%%%%%%%%%%%%%%
\subsection{Discretization in the parameter domain and parametric enrichment}
\label{sec:parameter_discretization}
%%%%%%%%%%%%%%%%%%%%

For all $m \in \N$,
let $(P_n^m)_{n\in\N_0}$ be the sequence of univariate polynomials
which are orthogonal to each other with respect to $\pi_m$ such that $P_n^m$ is a polynomial of degree
$n \in \N_0$ with $\norm{P_n^m}{L^2_{\pi_m}(-1,1)}=1$ and $P_0^m \equiv 1$.
It is well-known that $\{P_n^m : n\in\N_0 \}$ constitutes an orthonormal basis of $L^2_{\pi _m}(-1,1)$.

Let $\N_0^\N := \{\nu = (\nu_m)_{m\in\N} : \nu_m\in\N_0 \text{ for all } m\in\N \}$.
We define the support of $\nu = (\nu_m)_{m\in\N} \in \N_0^\N$ as $\supp(\nu):= \{m\in\N : \nu_m\neq 0 \}$.
We denote by
$ \III := \{\nu \in \N_0^{\N} : \#\supp(\nu) < \infty \}$ the set of all finitely supported elements of $\N_0^{\N}$.
Note that $\III$ is countable.
With each $\nu \in \III$, we associate the multivariate polynomial $P_\nu$ given by
\begin{equation*}
P_\nu(\y)
:= \prod_{m\in\N} P_{\nu_m}^m(y_m)
= \prod_{m\in\supp(\nu)} P_{\nu_m}^m(y_m)
\quad \text{for all } \nu \in \III \text{ and all } \y \in \Gamma.
\end{equation*}
It is well-known that
the set $\{ P_\nu : \nu\in\III \}$ is an orthonormal basis of $L^2_\pi(\Gamma)$;
see, e.g., \cite[Theorem~2.12]{sg11}.

Our discretization in the parameter domain will be based on an \emph{index set} $\PPP_\coarse$,
i.e.,
a finite subset of $\III$.
We denote by $\supp(\PPP_\coarse) := \bigcup_{\nu \in \PPP_\coarse} \supp(\nu)$
the set of active parameters in~$\PPP_\coarse$.
We denote by $\0 = (0,0,\dots)$ the zero index
and consider the initial index set $\PPP_0 := \{\0\}$.
Throughout this work,
we assume that all index sets employed for the discretization in the parameter domain
contain the zero index, i.e.,
there holds $\PPP_0 \subseteq \PPP_\coarse$
for each index set $\PPP_\coarse$.
Following~\cite{bprr18+}, for a fixed $\overline M \in \N$, we introduce the \emph{detail index set}
\begin{equation} \label{def:Q}
\QQQ_\coarse
:= \{ \mu \in \III \setminus \PPP_\coarse : \mu = \nu \pm \eps_m
\text{ for some } \nu \in \PPP_\coarse \text{ and some } m = 1,\dots, M_{\PPP_\coarse} + \overline M \}.
\end{equation}
Here,
$M_{\PPP_\coarse} := \#\supp(\PPP_\coarse) \in \N_0$ is the number of active parameters in $\PPP_\coarse$,
$\overline M$ represents the number of newly activated parameters in $\QQQ_\coarse$,
while,
for any $m \in \N$, $\eps_m \in \III$ denotes the $m$-th unit sequence,
i.e., $(\eps_m)_i = \delta_{mi}$ for all $i \in \N$.
In what follows, to simplify the presentation, we set $\overline M = 1$ in~\eqref{def:Q}.
A~parametric enrichment of $\PPP_\coarse$ is obtained by adding to it some marked indices
$\MMM_\coarse \subseteq \QQQ_\coarse$,
i.e., $\PPP_\fine := \PPP_\coarse \cup \MMM_\coarse$.
Clearly,
$\PPP_\coarse \subseteq \PPP_\fine \subseteq \PPP_\coarse \cup \QQQ_\coarse$, where
at least one of the inclusions is strict.

%%%%%%%%%%%%%%%%%%%%
\subsection{Multilevel approximation and multilevel refinement}
%%%%%%%%%%%%%%%%%%%%

We start by observing that the Bochner space $\V = L^2_\pi(\Gamma;\X)$
is isometrically isomorphic to
$\X \otimes L^2_\pi(\Gamma)$ and that
each function $\vv \in \V$ can be represented in the form
\begin{equation}\label{eq:representation}
 \vv(x,\y) = \sum_{\nu \in \III} v_\nu(x) P_\nu(\y)
 \quad\text{with unique coefficients}\
 v_\nu \in \X.
\end{equation}
A finite-dimensional subspace of $\V$ can be obtained by considering functions
with a similar representation,
where the infinite sum in~\eqref{eq:representation} is truncated
to a finite index set and the coefficients $v_\nu \in \X$ are approximated in suitable finite element spaces.
To this end,
let $\P_\coarse = [\PPP_\coarse, (\TT_{\coarse\nu})_{\nu \in \III}]$ be a
\emph{multilevel structure}~\cite{bpr2021+},
consisting of
a finite index set $\PPP_\coarse \subset \III$
and a family of meshes $(\TT_{\coarse\nu})_{\nu \in \III}$,
where $\TT_{\coarse\nu} \in \refine(\TT_0)$ for all $\nu \in \PPP_\coarse$,
while $\TT_{\coarse\nu} = \TT_0$ for all $\nu \in \III \backslash \PPP_\coarse$.

For two multilevel structures
$\P_\coarse = [\PPP_\coarse, (\TT_{\coarse\nu})_{\nu \in \III}]$
and
$\P_\fine = [\PPP_\fine, (\TT_{\fine\nu})_{\nu \in \III}]$,
we say that
$\P_\fine$ is obtained from $\P_\coarse$ using
one step of \emph{multilevel refinement},
and we write $\P_\fine = \Refine(\P_\coarse, \M_\coarse)$,
if the following conditions are satisfied:
\begin{itemize}
\item $\M_\coarse = [\MMM_\coarse, (\MM_{\coarse\nu})_{\nu \in \PPP_\coarse}]$
with $\MMM_\coarse \subseteq \QQQ_\coarse$ and
$\MM_{\coarse\nu} \subseteq \NN_{\coarse\nu}^+$ for all $\nu \in \PPP_\coarse$;
\item $\PPP_\fine = \PPP_\coarse \cup \MMM_\coarse$;
\item for all $\nu \in \PPP_\coarse$, there holds $\TT_{\fine\nu} = \refine(\TT_{\coarse\nu}, \MM_{\coarse\nu})$;
\item for all $\nu \in \III \backslash \PPP_\coarse$, there holds $\TT_{\fine\nu} = \TT_{\coarse\nu} = \TT_0$.
\end{itemize}
Mimicking
the notation in subsections~\ref{sec:spatial_discretization}--\ref{sec:parameter_discretization},
we consider the initial multilevel structure $\P_0 := [\PPP_0, (\TT_{0\nu})_{\nu \in \III}]$
consisting of the initial index set $\PPP_0$ and such that
$\TT_{0\nu} = \TT_0$ for all $\nu \in \III$.
We denote by $\Refine(\P_0)$ the set of all multilevel structures obtained from $\P_0$
by performing finitely many steps of multilevel refinement.
Throughout this work,
we assume that all multilevel structures
employed to construct a finite-dimensional subspace of $\V$
belong to $\Refine(\P_0)$.

Given a multilevel structure $\P_\coarse = [\PPP_\coarse, (\TT_{\coarse\nu})_{\nu \in \III}]$,
let $\X_{\coarse\nu} = \SS^1_0(\TT_{\coarse\nu})$ for all $\nu \in \PPP_\coarse$.
We consider the multilevel approximation space
\begin{equation}
\label{eq:def:V}
 \V_\coarse := \bigoplus_{\nu \in \PPP_\coarse} \V_{\coarse\nu} \subset \V
 \ \; \text{with} \ \;
 \V_{\coarse\nu} := \X_{\coarse\nu} \otimes {\rm span}\{P_\nu\}
 = {\rm span}\set{\varphi_{\coarse\nu,\xi} P_\nu}{\xi \in \NN_{\coarse\nu} \setminus \partial D}.
\end{equation}
Note that 
$\dim\V_\coarse = \sum_{\nu \in \PPP_\coarse}\dim\X_{\coarse\nu}$,
i.e.,
$\V_\coarse$ is a finite-dimensional subspace of $\V$,
and that each function $\vv_\coarse \in \V_\coarse$ can be represented in the form
(cf.~\eqref{eq:representation})
\begin{equation*}
 \vv_\coarse(x,\y) = \sum_{\nu \in \PPP_\coarse} v_{\coarse\nu}(x) P_\nu(\y)
 \quad\text{with unique coefficients}\
 v_{\coarse\nu} \in \X_{\coarse\nu}.
\end{equation*}
Moreover, by construction,
multilevel refinement implies nestedness of the associated multilevel spaces, i.e.,
if $\P_\fine \in \Refine(\P_\coarse)$ then $\V_\coarse \subseteq \V_\fine$.

For the multilevel approximation space~$\V_\coarse$ associated with
any given multilevel structure $\P_\coarse = [\PPP_\coarse, (\TT_{\coarse\nu})_{\nu \in \III}]$,
we consider the enriched subspace $\widehat\V_\coarse \subset \V$ defined as
\begin{equation}\label{eq:def:Vhat}
\widehat\V_\coarse 
:= \left(\bigoplus_{\nu \in \PPP_\coarse} \big[ \widehat\X_{\coarse\nu} \otimes {\rm span}\{P_\nu\} \big]\right)
\oplus
\left(\bigoplus_{\nu \in \QQQ_\coarse} \big[ \X_0 \otimes {\rm span}\{P_\nu\} \big]\right).
\end{equation}
Note that $\V_\coarse \subseteq \V_\fine \subseteq \widehat\V_\coarse$
for any $\P_\fine = \Refine(\P_\coarse, \M_\coarse)$. 
Moreover, $\widehat\V_\coarse$ corresponds to the multilevel structure
$\widehat\P_\coarse = \Refine(\P_\coarse, \M_\coarse)$
with $\M_\coarse = [\QQQ_\coarse, (\NN_{\coarse\nu}^+)_{\nu \in \PPP_\coarse}]$.

%%%%%%%%%%%%%%%%%%%%
\subsection{Multilevel SGFEM approximation}
%%%%%%%%%%%%%%%%%%%%

Given an arbitrary $\ww \in \V$, let $\ww_\coarse \in \V_\coarse$
and $\widehat\ww_\coarse \in \widehat\V_\coarse$ denote the Galerkin projections of $\ww$
onto $\V_\coarse$ and $\widehat\V_\coarse$, respectively,
i.e.,
\begin{subequations}
\begin{align}
\label{eq:galerkin:aux}
 B(\ww_\coarse, \vv_\coarse) &= B(\ww, \vv_\coarse) 
 \quad \text{for all } \vv_\coarse \in \V_\coarse,\\
\label{eq:galerkin:aux2}
 B(\widehat\ww_\coarse, \widehat\vv_\coarse) &= B(\ww, \widehat\vv_\coarse) 
 \quad \text{for all } \widehat\vv_\coarse \in \widehat\V_\coarse.
\end{align}
\end{subequations}
Existence and uniqueness of both $\ww_\coarse \in \V_\coarse$
and $\widehat \ww_\coarse \in \widehat\V_\coarse$ follow from the Riesz theorem.
Moreover,
there holds the so-called Galerkin orthogonality
\begin{equation} \label{eq:galerkin-orthogonality}
B(\ww - \ww_\coarse, \vv_\coarse ) = 0 
\quad \text{for all } \vv_\coarse \in \V_\coarse
\end{equation}
as well as
the best approximation property
\begin{equation} \label{eq:best_app}
\bnorm{\ww - \ww_\coarse} = \min_{\vv_\coarse \in \V_\coarse} \bnorm{\ww - \vv_\coarse}
\end{equation}
(the same properties clearly hold also for $\widehat\ww_\coarse$
with $\V_\coarse$ replaced by $\widehat\V_\coarse$).
Furthermore, since $\V_\coarse \subset \widehat\V_\coarse$, there holds
\begin{equation} \label{eq:Pythagoras}
\bnorm{\ww - \ww_\coarse}^2
=
\bnorm{\ww - \widehat\ww_\coarse}^2
+
\bnorm{\ww_\coarse - \widehat\ww_\coarse}^2.
\end{equation}
In particular,
\begin{equation}\label{eq:saturation_aux}
\bnorm{\ww - \widehat\ww_\coarse} \le \bnorm{\ww - \ww_\coarse}
\quad
\text{and}\quad
\bnorm{\ww_\coarse - \widehat\ww_\coarse} \le \bnorm{\ww - \ww_\coarse}
.
\end{equation}
We define the multilevel SGFEM approximation $\uu_\coarse \in \V_\coarse$
of the solution $\uu \in \V$ to the primal problem~\eqref{eq:weakform}
as the Galerkin projection of $\uu$ onto the multilevel approximation space~$\V_\coarse$.
Equivalently, $\uu_\coarse \in \V_\coarse$ can be characterized
as
the unique solution of the following discrete variational problem:
Find $\uu_\coarse \in \V_\coarse$ such that
\begin{equation}\label{eq:discrete_formulation}
 B(\uu_\coarse, \vv_\coarse) = F(\vv_\coarse) 
 \quad \text{for all } \vv_\coarse \in \V_\coarse. 
\end{equation}

%%%%%%%%%%%%%%%%%%%%
\subsection{\textsl{A~posteriori} error estimation}
\label{sec:error_estimator}
%%%%%%%%%%%%%%%%%%%%

To obtain computable estimates of the energy error
$\bnorm{\ww - \ww_\coarse}$ of the Galerkin projection,
we follow the approach proposed in~\cite{bpr2020+},
which is based on the separate estimation
of the error components associated with
discretizations in physical and parameter domains.

Here and in the sequel, for the sake of brevity,
we denote the inner product on $\X = H^1_0(D)$
by $\dual{w}{v} := \int_D a_0\nabla w \cdot \nabla v \, \d{x}$
and the induced energy norm
by $\norm{\cdot}{D} := \norm{a_0^{1/2}\nabla(\cdot)}{L^2(D)}$.

The parametric components of the error in the Galerkin approximation $\ww_\coarse$ % of $\ww$
are estimated using the hierarchical error indicators
\begin{subequations} \label{eq1:parametric-error-estimate}
\begin{equation} \label{eq1:parametric-error-estimate:a}
 \tau_\coarse(\ww|\nu) := \norm{e_{\ww|\coarse\nu}}{D}
 \quad \text{for all } \nu \in \QQQ_\coarse, 
\end{equation}
where $e_{\ww|\coarse\nu} \in \X_0$ is the unique solution of
\begin{equation} \label{eq1:parametric-error-estimate:b}
 \dual{e_{\ww|\coarse\nu}}{v_0}
 = B(\ww - \ww_\coarse, v_0 P_\nu) 
  \quad \text{for all } v_0  \in \X_0.
\end{equation}%
\end{subequations}%
The errors attributable to spatial discretizations are
estimated using the \emph{two-level} error indicators
\begin{equation}\label{eq1:spatial-error-estimate}
 \tau_{\coarse}(\ww|\nu,\xi)
 := \frac{|B(\ww-\ww_\coarse,\widehat\varphi_{\coarse\nu,\xi} P_\nu)|}{\norm{\widehat\varphi_{\coarse\nu,\xi}}{D}}
 \quad \text{for all } \nu \in \PPP_\coarse \text{ and all } \xi \in \NN_{\coarse\nu}^+.
\end{equation}
Overall, we thus consider the {\sl a~posteriori} error estimate
\begin{equation}\label{eq:def:tau}
 \tau_\coarse(\ww)
 := \Bigg( \sum_{\nu \in \PPP_\coarse}\sum_{\xi \in \NN_{\coarse\nu}^+} \tau_{\coarse}(\ww|\nu,\xi)^2
 + \sum_{\nu \in \QQQ_\coarse} \tau_\coarse(\ww|\nu)^2  \Bigg)^{1/2}.
\end{equation}

%%%%%%%%%%%%%%%%%%%%
\begin{remark}
Note that, for a general unknown $\ww \in \V$, the error estimate $\tau_\coarse(\ww)$ is not computable.
However, we shall employ the estimate $\tau_\coarse(\ww)$ only for $\ww \in \{\uu, \zz[\uu_\bullet]\}$,
where $\uu \in \V$ is the solution to
the primal problem~\eqref{eq:weakform},
while $\zz[\cdot] \in \V$ denotes the solution to the so-called dual problem
(see,~\eqref{eq:nonlinear:weakform:dual} below).
For these choices of $\ww$,
one can evaluate $B(\ww-\ww_\coarse,\cdot)$ in~\eqref{eq1:parametric-error-estimate}--\eqref {eq1:spatial-error-estimate},
so that $\tau_\coarse(\ww)$ becomes fully computable;
e.g., for the primal solution $\ww = \uu$ and
$\vv \in \{ v_0 P_\nu , \widehat\varphi_{\coarse\nu,\xi} P_\nu\}$,
one has $B(\uu - \uu_\coarse, \vv) = F(\vv) - B(\uu_\coarse, \vv)$.
\end{remark}
%%%%%%%%%%%%%%%%%%%%

It follows from the first inequality in~\eqref{eq:saturation_aux} that
the error of the Galerkin projection associated with the enriched multilevel space
$\widehat\V_\coarse$ is not larger than the one for the space~$\V_\coarse$.
We say that the \emph{saturation assumption} is satisfied,
if considering the enriched multilevel space leads to a uniform strict reduction of the best approximation error,
i.e., if there exists a constant $0 < \qsat < 1$ such that 
\begin{equation}\label{eq:saturation}
 \bnorm{\ww - \widehat\ww_\coarse} 
 \le \qsat \, \bnorm{\ww - \ww_\coarse}.
\end{equation}
We note that the adaptive multilevel SGFEM algorithm that ensures error reduction by a uniform factor at each refinement step
has recently been designed and analyzed in~\cite{BachmayrEEV_CAF}, albeit for a very specific type of the expansion in~\eqref{eq1:a}
written in terms of hierarchical and locally supported expansion coefficients and
for a parametric enrichment strategy associated with this expansion.

We now recall the following main result from~\cite{bpr2020+},
which shows the equivalence of the error estimate $\tau_\coarse(\ww)$ in~\eqref{eq:def:tau}
to the \emph{error reduction} $\bnorm{\ww_\coarse - \widehat\ww_\coarse}$.
This implies that the proposed error estimator is \emph{efficient}, i.e.,
up to a multiplicative constant,
it provides a lower bound for the energy norm of the error,
while its \emph{reliability}
(i.e., the upper bound for the error)
is equivalent to the saturation assumption~\eqref{eq:saturation}.

%%%%%%%%%%%%%%%%%%%%
\begin{theorem}[{\cite[Theorem~2]{bpr2020+}}]\label{thm:estimator}
Let $d \in \{2, 3\}$ and $\ww \in \V$.
For the multilevel structures $\P_\coarse, \widehat\P_\coarse \in \Refine(\P_0)$,
consider the multilevel approximation spaces $\V_\coarse \subseteq \widehat\V_\coarse$
with the associated Galerkin solutions $\ww_\coarse \in \V_\coarse$
(solving~\eqref{eq:galerkin:aux}) and
$\widehat\ww_\coarse \in \widehat\V_\coarse$ (solving~\eqref{eq:galerkin:aux2}).
Then, there holds
\begin{equation}\label{eq1:thm:estimator}
 \Cest^{-1} \, \bnorm{\widehat \ww_\coarse - \ww_\coarse}
 \le \tau_\coarse(\ww) 
 \le \Cest \, \bnorm{\widehat \ww_\coarse - \ww_\coarse}. 
\end{equation} 
Furthermore,
under the saturation assumption~\eqref{eq:saturation},
the estimates~\eqref{eq1:thm:estimator} are equivalent to
\begin{equation}\label{eq2:thm:estimator}
 \frac{(1-\qsat^2)^{1/2}}{\Cest} \, \bnorm{\ww - \ww_\coarse}
 \le \tau_\coarse(\ww)
 \le \Cest \, \bnorm{\ww - \ww_\coarse},
\end{equation}
The constant $\Cest \ge 1$ in~\eqref{eq1:thm:estimator}--\eqref{eq2:thm:estimator}
depends only on $\TT_0$,
the mean field $a_0$, and the constants $\lambda, \Lambda > 0$ in~\eqref{eq:lambda}.
\qed
\end{theorem}
%%%%%%%%%%%%%%%%%%%%

%%%%%%%%%%%%%%%%%%%%
\section{Goal-oriented adaptive SGFEM with nonlinear goal functional}
\label{sec:goafem_nonlinear}
%%%%%%%%%%%%%%%%%%%%

In this section, for a class of (possibly nonlinear) goal functionals $\gbold \colon \V \to \R$,
we develop a goal-oriented error estimation strategy,
design the associated adaptive algorithm with multilevel SGFEM approximations, and perform its convergence analysis.

%%%%%%%%%%%%%%%%%%%%
\subsection{Dual problem and goal-oriented error estimate}
%%%%%%%%%%%%%%%%%%%%

Let the goal functional $\gbold \colon\! \V \to \R$ be in $C^1$,
in the sense that it is G\^ateaux differentiable
and its G\^ateaux derivative $\gbold' \colon \V \to \V^*$ is continuous.
In addition, we suppose that there exists $\Cgoal \ge 0$ such that
\begin{align}\label{eq:nonlinear:ass:goal}
   |\Dual{\gbold'(\vv) - \gbold'(\ww)}{\zz}|
   \le \Cgoal \, \bnorm{\vv-\ww} \bnorm{\zz}
   \quad \text{for all } \vv, \ww, \zz \in \V,
\end{align}
where $\Dual{\cdot}{\cdot}$ denotes the duality pairing between $\V$ and its dual $\V^*$.
Following the approach adopted in~\cite{bprr18+} for the case of linear goal functionals,
we aim to formulate a dual problem which allows to derive a goal-oriented error estimate.

The fundamental theorem of calculus proves that
\begin{equation}\label{eq:nonlinear:goal:dual:exact}
 \gbold(\uu) - \gbold(\uu_\coarse) 
 = \int_0^1 \Dual{\gbold'(\uu_\coarse + t \, (\uu - \uu_\coarse))}{\uu - \uu_\coarse} \, \d{t}
 =: \Dual{\gbold^\star_{\uu}(\uu_\coarse)}{\uu - \uu_\coarse}.
\end{equation}
This identity suggests to consider a dual problem with right-hand side given by $\gbold^\star_{\uu}(\uu_\coarse) \in \V^\star$.
However, $\gbold^\star_{\uu}(\uu_\coarse)$ depends also on the unknown solution $\uu$ 
and thus cannot be used to formulate a practical dual problem.
Observing that formally $\|\gbold^\star_{\uu}(\uu_\coarse) - \gbold'(\uu_\coarse)\|_{\V^\star} \to 0$ as $\uu_\coarse \to \uu$,
for a given $\ww \in \V$
(in what follows, $\ww \in \{ \uu, \uu_\coarse \}$),
we consider the following (practical, if $\ww$ is known) dual problem:
Find $\zz[\ww] \in \V$ such that
\begin{equation}\label{eq:nonlinear:weakform:dual}
 B(\vv,\zz[\ww]) = \Dual{\gbold'(\ww)}{\vv}
 \quad \text{for all } \vv \in \V.
\end{equation}
Later, we will approximate $\zz[\ww] \in \V$ by its Galerkin projection $\zz_\coarse[\ww] \in \V_\coarse$, i.e., 
\begin{equation}\label{eq:nonlinear:weakform:dual:discrete}
 B(\vv_\coarse,\zz_\coarse[\ww]) = \Dual{\gbold'(\ww)}{\vv_\coarse}
 \quad \text{for all } \vv_\coarse \in \V_\coarse.
\end{equation}
Existence and uniqueness of both $\zz[\ww] \in \V$ and $\zz_\coarse[\ww] \in \V_\coarse$ follow from the Riesz theorem.
Note that
\begin{equation*}
\begin{split}
 \gbold(\uu) - \gbold(\uu_\coarse) 
 &\reff{eq:nonlinear:goal:dual:exact}=
 \Dual{\gbold^\star_{\uu}(\uu_\coarse) - \gbold'(\uu_\coarse)}{\uu - \uu_\coarse} + \Dual{\gbold'(\uu_\coarse)}{\uu - \uu_\coarse}
 \\& %\quad
 \reff{eq:nonlinear:weakform:dual}=
 \Dual{\gbold^\star_{\uu}(\uu_\coarse) - \gbold'(\uu_\coarse)}{\uu - \uu_\coarse} + B( \uu - \uu_\coarse, \zz[\uu_\coarse] )
 \\& %\quad
 \reff{eq:galerkin-orthogonality}=
 \Dual{\gbold^\star_{\uu}(\uu_\coarse) - \gbold'(\uu_\coarse)}{\uu - \uu_\coarse} +
 B( \uu - \uu_\coarse, \zz[\uu_\coarse] - \zz_\coarse[\uu_\coarse] ),
\end{split}
\end{equation*}
and the first term on the right-hand side can be estimated as follows:
\begin{align*}
 \Dual{\gbold^\star_{\uu}(\uu_\coarse) - \gbold'(\uu_\coarse)}{\uu - \uu_\coarse}
 &\reff{eq:nonlinear:goal:dual:exact}= \int_0^1 \Dual{[\gbold'(\uu_\coarse + t(\uu-\uu_\coarse))
                                - \gbold'(\uu_\coarse)]}{\uu - \uu_\coarse} \, \d{t}
 \\
 &\reff{eq:nonlinear:ass:goal}\le \Cgoal \int_0^1 t
 \bnorm{\uu - \uu_\coarse}^2\, \d{t}
 = \frac{1}{2} \, \Cgoal \, 
 \bnorm{\uu - \uu_\coarse}^2.
\end{align*}
Hence, we derive the following estimate of the error in the nonlinear goal~functional:
\begin{equation}\label{eq:nonlinear:goal-error-estimate}
 | \gbold(\uu) - \gbold(\uu_\coarse) | 
 \le \bnorm{\uu - \uu_\coarse} \bnorm{\zz[\uu_\coarse] - \zz_\coarse[\uu_\coarse]} + \frac{1}{2} \,\Cgoal \, \bnorm{\uu-\uu_\coarse}^2.
\end{equation}

%%%%%%%%%%%%%%%%%%%%
\begin{remark}
The key assumption~\eqref{eq:nonlinear:ass:goal}, and hence the error estimate~\eqref{eq:nonlinear:goal-error-estimate},
is valid at least for linear and quadratic goal functionals.
First, for a bounded linear goal functional $\gbold \in \V^*$, one has
$\Dual{\gbold'(\vv)}{\zz} = \Dual{\gbold}{\zz}$ for all $\vv, \zz \in \V$.
Hence, the dual problem in~\eqref{eq:nonlinear:weakform:dual} simplifies to the following:
Find $\zz \in \V$ such that
%\begin{equation}\label{eq:nonlinear:weakform:dual}
$ B(\vv,\zz) = \Dual{\gbold}{\vv} $
for all $\vv \in \V$.
%\end{equation}
Furthermore, inequality~\eqref{eq:nonlinear:ass:goal} is satisfied with $\Cgoal = 0$ and
the error estimate~\eqref{eq:nonlinear:goal-error-estimate} reduces to the following (cf.~\cite[section~1.1]{bprr18+}):
\[
   \lvert \gbold(\uu) - \gbold(\uu_\coarse) \rvert
   \le \bnorm{\uu - \uu_\coarse} \, \bnorm{\zz - \zz_\coarse}.
\]
Second, consider the quadratic goal functional $\gbold(\uu) = b(\uu,\uu)$, where $b \colon \V \times \V \to \R$ is a continuous bilinear form.
Then, $\Dual{\gbold'(\vv)}{\zz} = b(\vv,\zz) + b(\zz,\vv)$, and it follows that
\begin{align*}
 \Dual{\gbold'(\vv) - \gbold'(\ww)}{\zz}
 = b(\vv-\ww,\zz) + b(\zz,\vv-\ww)
 \quad
 \forall\, \vv, \ww, \zz \in \V.
\end{align*}%
Hence,~\eqref{eq:nonlinear:ass:goal} is satisfied and $\Cgoal > 0$ depends only on
the continuity constant for $b(\cdot,\cdot)$.
\end{remark}
%%%%%%%%%%%%%%%%%%%%

The following lemma will later turn out to be a crucial argument.

%%%%%%%%%%%%%%%%%%%%
\begin{lemma}\label{prop:nonlinear:goal:aux:stability}
For any $\ww \in \V$, there holds
\begin{align}\label{eq:nonlinear:goal:aux:stability}
 \bnorm{\zz_\coarse[\ww] - \zz_\coarse[\uu_\coarse]} 
 \le \bnorm{\zz[\ww] - \zz[\uu_\coarse]}
 \le \Cgoal \, \bnorm{\ww - \uu_\coarse}.
\end{align}
\end{lemma}
%%%%%%%%%%%%%%%%%%%%

%%%%%%%%%%%%%%%%%%%%
\begin{proof}
First, it follows from~\eqref{eq:nonlinear:weakform:dual}--\eqref{eq:nonlinear:weakform:dual:discrete} that
\begin{align*}
 \bnorm{\zz_\coarse[\ww] &- \zz_\coarse[\uu_\coarse]}^2 
 = B(\zz_\coarse[\ww] - \zz_\coarse[\uu_\coarse], \zz_\coarse[\ww]) 
 - B(\zz_\coarse[\ww] - \zz_\coarse[\uu_\coarse], \zz_\coarse[\uu_\coarse] )
 \\
 &\refp{eq:nonlinear:weakform:dual:discrete}= B(\zz_\coarse[\ww] - \zz_\coarse[\uu_\coarse], \zz[\ww]) 
 - B(\zz_\coarse[\ww] - \zz_\coarse[\uu_\coarse], \zz[\uu_\coarse] )
 \\
 &\refp{eq:nonlinear:weakform:dual}=B(\zz_\coarse[\ww] - \zz_\coarse[\uu_\coarse], \zz[\ww] - \zz[\uu_\coarse]) 
 \le \bnorm{\zz_\coarse[\ww] - \zz_\coarse[\uu_\coarse]} \bnorm{\zz[\ww] - \zz[\uu_\coarse]},
\end{align*}
which proves the first estimate. The second estimate follows from~\eqref{eq:nonlinear:weakform:dual} and~\eqref{eq:nonlinear:ass:goal}, namely
\begin{align*}
 \bnorm{\zz[\ww] &- \zz[\uu_\coarse]}^2 
 = B(\zz[\ww] - \zz[\uu_\coarse], \zz[\ww])
 - B(\zz[\ww] - \zz[\uu_\coarse], \zz[\uu_\coarse])
 \\&\reff{eq:nonlinear:weakform:dual}= \Dual{\gbold'(\ww)}{\zz[\ww] - \zz[\uu_\coarse]}
 - \Dual{\gbold'(\uu_\coarse)}{\zz[\ww] - \zz[\uu_\coarse]}
 \\&\refp{eq:nonlinear:weakform:dual}= \Dual{\gbold'(\ww) - \gbold'(\uu_\coarse)}{\zz[\ww] - \zz[\uu_\coarse]}
 \reff{eq:nonlinear:ass:goal}\le \Cgoal \bnorm{\ww - \uu_\coarse} \bnorm{\zz[\ww] - \zz[\uu_\coarse]}.
\end{align*}
This concludes the proof.
\end{proof}
%%%%%%%%%%%%%%%%%%%%

To estimate the energy errors appearing on the right-hand side of~\eqref{eq:nonlinear:goal-error-estimate},
we consider the error estimation strategy introduced in section~\ref{sec:error_estimator}.
In the rest of the paper, unless otherwise specified, we use the abbreviated notation
\begin{equation}\label{eq:nonlinear:tau:abbreviations}
 \mu_\coarse := \tau_\coarse(\uu)
 \quad \text{ and } \quad
 \zeta_\coarse := \tau_\coarse(\zz[\uu_\coarse]). 
\end{equation}
The same notation will be used for local contributions to the error estimates,~i.e.,
\begin{equation} \label{eq:nonlinear:tau:abbreviations:local}
   \mu_\coarse(\nu) := \tau_\coarse(\uu|\nu),\ 
   \mu_\coarse(\nu,\xi) :=\tau_\coarse(\uu|\nu,\xi)
   \ \text{and}\
   \zeta_\coarse(\nu) := \tau_\coarse(\zz[\uu_\coarse]|\nu),\ 
   \zeta_\coarse(\nu,\xi) :=\tau_\coarse(\zz[\uu_\coarse]|\nu,\xi).
\end{equation}
We emphasize that both error estimates in~\eqref{eq:nonlinear:tau:abbreviations},
as well as their local contributions, are indeed computable.
Combining the error estimate~\eqref{eq:nonlinear:goal-error-estimate} with
the results of Theorem~\ref{thm:estimator} and Lemma~\ref{prop:nonlinear:goal:aux:stability},
we obtain a reliable \textsl{a~posteriori} error estimate of the error in the nonlinear goal~functional.
We emphasize that the constant $\Crel$ in the following estimate~\eqref{eq1:nonlinear:cor:goal} depends only
on the saturation assumption~\eqref{eq:saturation} for $\ww = \uu$ and $\ww = \zz[\uu]$,
while any dependence on $\zz[\uu_\coarse]$ as used in the definition of $\zeta_\coarse$ is avoided.

%%%%%%%%%%%%%%%%%%%%
\begin{proposition}\label{prop:eq1:nonlinear:cor:goal}
Let $d \in \{2, 3\}$.
Suppose the saturation assumption~\eqref{eq:saturation} for both the primal solution $\ww = \uu$ to~\eqref{eq:weakform}
and the (theoretical) dual solution $\ww = \zz[\uu]$.
Then, there holds the \textsl{a~posteriori} goal-oriented error estimate
\begin{align}\label{eq1:nonlinear:cor:goal}
 |\gbold(\uu) - \gbold(\uu_\coarse) | 
 \le \Crel \,\mu_\coarse \, \big[ \mu_\coarse^2 + \zeta_\coarse^2 \big]^{1/2},
\end{align}
where $\Crel>0$ depends only
on the constants $\Cgoal \ge 0$ in~\eqref{eq:nonlinear:ass:goal},
$\Cest \ge 1$ in~\eqref{eq1:thm:estimator},
and $0 < \qsat < 1$ in~\eqref{eq:saturation}.
\end{proposition}
%%%%%%%%%%%%%%%%%%%%

%%%%%%%%%%%%%%%%%%%%
\begin{proof}
Under the saturation assumption~\eqref{eq:saturation} for $\ww = \uu$, Theorem~\ref{thm:estimator} proves that
\begin{align}\label{eq1:nonlinear:proof:cor:goal}
 \bnorm{\uu - \uu_\coarse} \le \frac{\Cest}{(1-\qsat^2)^{1/2}} \, \mu_\coarse.
\end{align}
Under the saturation assumption~\eqref{eq:saturation} for $\ww = \zz[\uu]$, Theorem~\ref{thm:estimator} proves that
\begin{align*}
 \bnorm{\zz[\uu_\coarse] - \zz_\coarse[\uu_\coarse]}
 &\preff{eq:nonlinear:goal:aux:stability}\le \bnorm{\zz[\uu] - \zz[\uu_\coarse]}
 + \bnorm{\zz[\uu] - \zz_\coarse[\uu]}
 + \bnorm{\zz_\coarse[\uu] - \zz_\coarse[\uu_\coarse]}
 \\
 &\reff{eq:nonlinear:goal:aux:stability}\le 
 2\Cgoal \, \bnorm{\uu - \uu_\coarse} + \bnorm{\zz[\uu] - \zz_\coarse[\uu]}
 \\
 &\reff{eq2:thm:estimator}\le
 2\Cgoal \, \bnorm{\uu - \uu_\coarse} + \frac{\Cest}{(1-\qsat^2)^{1/2}} \, \tau_\coarse(\zz[\uu]).
\end{align*}
Note that the \textsl{a~posteriori} error estimate has a seminorm structure and, hence,
it satisfies the triangle inequality. Therefore,
\begin{align*}
 |  \tau_\coarse(\zz[\uu]) - \zeta_\coarse |
 &\le \tau_\coarse(\zz[\uu]-\zz[\uu_\coarse])
 \reff{eq1:thm:estimator}\le \Cest \,  \bnorm{(\widehat \zz_\coarse[\uu] - \widehat\zz_\coarse[\uu_\coarse]) - ( \zz_\coarse[\uu] - \zz_\coarse[\uu_\coarse])}
 \\
 &\le \Cest \, \big( \bnorm{\widehat \zz_\coarse[\uu] - \widehat\zz_\coarse[\uu_\coarse]} + \bnorm{\zz_\coarse[\uu] - \zz_\coarse[\uu_\coarse]} \big)
 \reff{eq:nonlinear:goal:aux:stability}\le 2\Cest\Cgoal \bnorm{\uu - \uu_\coarse}.
\end{align*}
Combining the last two estimates, we derive that
\begin{align}
 \bnorm{\zz[\uu_\coarse] - \zz_\coarse[\uu_\coarse]} 
 &\preff{eq1:nonlinear:proof:cor:goal}\le 2\Cgoal \, \Big( 1  + \frac{\Cest^2}{(1-\qsat^2)^{1/2}} \Big) \,
 \bnorm{\uu - \uu_\coarse} + \frac{\Cest}{(1-\qsat^2)^{1/2}} \,\zeta_\coarse
 \nonumber
 \\
 &\reff{eq1:nonlinear:proof:cor:goal}\le
  \frac{\Cest}{(1-\qsat^2)^{1/2}} \, \bigg[ 2\Cgoal \, \Big( 1  + \frac{\Cest^2}{(1-\qsat^2)^{1/2}} \Big) \, 
  \mu_\coarse + \zeta_\coarse \bigg].
  \label{eq2:nonlinear:proof:cor:goal}
\end{align}
Overall, we thus see that
\begin{align*}
 | \gbold(\uu) - \gbold(\uu_\coarse) | 
 &\reff{eq:nonlinear:goal-error-estimate}\le \bnorm{\uu - \uu_\coarse} \bnorm{\zz[\uu_\coarse] - \zz_\coarse[\uu_\coarse]} + \frac{3}{2} \Cgoal \, \bnorm{\uu-\uu_\coarse}^2
 \\
 &\preff{eq:nonlinear:goal-error-estimate}\lesssim \mu_\coarse \big[ \mu_\coarse + \zeta_\coarse \big] + \mu_\coarse^2
\simeq \mu_\coarse \big[ \mu_\coarse^2 + \zeta_\coarse^2 \big]^{1/2}.
\end{align*}
This concludes the proof of~\eqref{eq1:nonlinear:cor:goal}.
\end{proof}
%%%%%%%%%%%%%%%%%%%%

%%%%%%%%%%%%%%%%%%%%
\subsection{Adaptive algorithm} \label{sec:alg_nonlinear}
%%%%%%%%%%%%%%%%%%%%

Our aim in this section is to extend
the adaptive SGFEM algorithm from~\cite{bpr2020+}
(see Algorithm~7.C therein) to the present goal-oriented setting for parametric PDEs.
On the one hand, following~\cite{bpr2020+},
the enhancement of the approximation space $\V_\ell$ for each $\ell \in \N_0$
is steered in Algorithm~\ref{algorithm} below by the D\"orfler marking criterion~\cite{doerfler}
performed on the joint set of all spatial and parametric error indicators (see steps~(iv)--(v)).
On the other hand, in view of the \textsl{a~posteriori} error estimate~\eqref{eq1:nonlinear:cor:goal},
we exploit the ideas proposed in~\cite{bip2021} in a much simpler non-parametric setting to
ensure that either the primal estimator $\mu_\ell$
or the combined primal-dual estimator $(\mu_\ell^2 + \zeta_\ell^2)^{1/2}$
tends to zero as~$\ell \to \infty$.

%%%%%%%%%%%%%%%%%%%%
\begin{algorithm}\label{algorithm}
{\bfseries Input:}
$\P_0 = [\PPP_0 , (\TT_{0\nu})_{\nu \in \III}]$ with
$\PPP_0 = \{ \0 \}$ and $\TT_{0\nu} := \TT_0$ for all $\nu \in \III$,
marking parameter $0 < \theta \le 1$.\\
{\bfseries Loop:}
For all $\ell = 0, 1, 2, \dots$, iterate the following steps:
\begin{itemize}
\item[\rm(i)] 
Compute the discrete primal solution $\uu_\ell \in \V_\ell$ and
the discrete dual solution $\zz_\ell[\uu_\ell] \in \V_\ell$ associated with $\P_\ell = [\PPP_\ell , (\TT_{\ell\nu})_{\nu \in \III}]$.
\item[\rm(ii)] 
For all $\nu \in \QQQ_\ell$, compute the parametric error indicators $\mu_\ell(\nu)$, $\zeta_\ell(\nu)$
given by~\eqref{eq:nonlinear:tau:abbreviations:local} and~\eqref{eq1:parametric-error-estimate}.
\item[\rm(iii)] 
For all $\nu \in \PPP_\ell$ and all $\xi \in \NN_{\ell\nu}^+$, compute
the spatial error indicators $\mu_\ell(\nu,\xi)$, $\zeta_\ell(\nu,\xi)$
given by~\eqref{eq:nonlinear:tau:abbreviations:local} and~\eqref{eq1:spatial-error-estimate}.
\item[\rm(iv)] 
Determine the sets $\MMM'_\ell \subseteq \QQQ_\ell$ and $\MM'_{\ell\nu} \subseteq \NN_{\ell\nu}^+$ for all $\nu \in \PPP_\ell$
such that
\begin{equation}\label{eq:doerfler:primal}
 \theta \, \mu_\ell^2
 \le \sum_{\nu \in \PPP_\ell} \sum_{\xi \in \MM'_{\ell\nu}} \mu_\ell(\nu,\xi)^2 + \sum_{\nu \in \MMM'_\ell} \mu_\ell(\nu)^2,
\end{equation}
where the overall cardinality $M_\ell' := \#\MMM'_\ell + \sum_{\nu \in \PPP_\ell}  \#\MM'_{\ell\nu}$ 
is minimal amongst all tuples $\M'_\ell = [\MMM'_\ell, (\MM'_{\ell\nu})_{\nu \in \PPP_\ell}]$
satisfying~\eqref{eq:doerfler:primal}.
\item[\rm(v)] 
Determine the sets $\MMM''_\ell \subseteq \QQQ_\ell$ and $\MM''_{\ell\nu} \subseteq \NN_{\ell\nu}^+$ for all $\nu \in \PPP_\ell$
such that
\begin{equation}\label{eq:nonlinear:doerfler:dual}
 \theta \, 
 \big[ \mu_\ell^2 + \zeta_\ell^2 \big]
 \le \sum_{\nu \in \PPP_\ell} \sum_{\xi \in \MM''_{\ell\nu}} \big[ \mu_\ell(\nu,\xi)^2 + \zeta_\ell(\nu,\xi)^2 \big]
 + \sum_{\nu \in \MMM''_\ell} \big[ \mu_\ell(\nu)^2 + \zeta_\ell(\nu)^2 \big].
\end{equation}
where the overall cardinality $M_\ell'' := \#\MMM''_\ell + \sum_{\nu \in \PPP_\ell}  \#\MM''_{\ell\nu}$ 
is minimal amongst all tuples $\M''_\ell = [\MMM''_\ell, (\MM''_{\ell\nu})_{\nu \in \PPP_\ell}]$
satisfying~\eqref{eq:nonlinear:doerfler:dual}.
\item[\rm(vi)] If $M_\ell' \le M_\ell''$, 
then choose $\MMM_\ell := \MMM'_\ell$ and $\MM_{\ell\nu} := \MM'_{\ell\nu}$ for all $\nu \in \PPP_\ell$. Otherwise choose $\MMM_\ell := \MMM''_\ell$ and $\MM_{\ell\nu} := \MM''_{\ell\nu}$ for all $\nu \in \PPP_\ell$.
\item[\rm(vii)] For all $\nu \in \PPP_\ell$, let $\TT_{\ell+1,\nu} := \refine(\TT_{\ell\nu},\MM_{\ell\nu})$.
\item[\rm(viii)] Define $\PPP_{\ell+1}
:= \PPP_\ell \cup \MMM_\ell$ and $\TT_{(\ell+1)\nu} := \TT_0$ for all $\nu \in \QQQ_{\ell+1}$.
\end{itemize}

{\bfseries Output:} For all $\ell \in \N_0$, the algorithm returns
the approximation $\gbold(\uu_\ell)$ of the goal functional $\gbold(\uu)$
and the associated goal-oriented error estimate $\mu_\ell \, \big[ \mu_\ell^2 + \zeta_\ell^2 \big]^{1/2}$.
\qed
\end{algorithm}
%%%%%%%%%%%%%%%%%%%%

We note that Algorithm~\ref{algorithm} can be seen as an extension of the goal-oriented adaptive algorithm from~\cite{bprr18+}
to the case of \emph{nonlinear} goal functionals and \emph{multilevel} SGFEM approximations.
%
%We also note that,
While the computations of the discrete primal and dual solutions in step~{\rm(i)} of Algorithm~\ref{algorithm}
can be carried out in parallel in the case of a linear goal functional $\gbold \in \V^*$
(in this case, the discrete primal and dual problems
are independent of each other),
in the nonlinear case they must be performed sequentially (first the primal problem, then the dual problem),
because the right-hand side of the discrete dual problem, i.e., \eqref{eq:nonlinear:weakform:dual:discrete} with $\ww = \uu_\coarse$,
depends on the discrete primal solution.

%%%%%%%%%%%%%%%%%%%%
\subsection{Convergence analysis}
%%%%%%%%%%%%%%%%%%%%

The following theorem is the main theoretical result of the present work.
Specifically, we prove
that Algorithm~\ref{algorithm} drives the
goal-oriented error estimates $\mu_\ell \, \big[ \mu_\ell^2 + \zeta_\ell^2 \big]^{1/2}$
to zero.
We emphasize that this result holds independently of the saturation assumption~\eqref{eq:saturation}.

%%%%%%%%%%%%%%%%%%%%
\begin{theorem}\label{thm:nonlinear:plain_convergence}
Let $d \in \{2, 3\}$.
For any choice of the marking parameter $0 < \theta \le 1$,
Algorithm~\ref{algorithm}
yields a convergent sequence of estimator products, i.e.,
$
 \mu_\ell \big[\mu_\ell^2 + \zeta_\ell^2\big]^{1/2} \xrightarrow{\ell \to \infty} 0.
$
\end{theorem}
%%%%%%%%%%%%%%%%%%%%

The following result is an immediate consequence of Theorem~\ref{thm:nonlinear:plain_convergence}
and the goal-oriented error estimate~\eqref{eq1:nonlinear:cor:goal} from
Proposition~\ref{prop:eq1:nonlinear:cor:goal}.

%%%%%%%%%%%%%%%%%%%%
\begin{corollary}\label{col:nonlinear:plain_convergence}
Let $d \in \{2, 3\}$.
Suppose that the saturation assumption~\eqref{eq:saturation} holds for both the primal solution $\ww \,{=}\, \uu$~and
the (theoretical) dual solution $\ww = \zz[\uu]$.
Then, for any choice of the marking parameter $0 < \theta \le 1$,
Algorithm~\ref{algorithm}
drives the error in the goal~functional to~zero, i.e.,
\begin{equation*}
 | \gbold(\uu) - \gbold(\uu_\ell) | 
 \le \Crel \, \mu_\ell \big[\mu_\ell^2 + \zeta_\ell^2\big]^{1/2}
 \xrightarrow{\ell \to \infty} 0.
\end{equation*}
\end{corollary}
%%%%%%%%%%%%%%%%%%%%

The proof of Theorem~\ref{thm:nonlinear:plain_convergence}
exploits the ideas from our own work~\cite{bprr18++}
on the convergence of adaptive \emph{single-level} SGFEM.
In the \emph{multilevel} framework for goal-oriented adaptivity, as considered in the present work,
the analysis needs to account for two distinctive aspects:
(i)~different spatial coefficients in the finite gPC-expansion (that represents the SGFEM solution)
may reside in different finite element spaces,
and
(ii)~the structure of the goal-oriented adaptive SGFEM
algorithm is inherently nonlinear
(due to the error bound being the product of two error estimates).
Therefore, we include full details of analysis where it addresses these two aspects
(cf. Proposition~\ref{prop:conv:spatial} and the proof of Theorem~\ref{thm:nonlinear:plain_convergence} below),
while referring to~\cite{bprr18++} for results that carry over from the single-level SGFEM~setting.

The first lemma is an early result from~\cite{bv84},
which proves that adaptive algorithms (without coarsening) always lead to convergence of the discrete solutions.

%%%%%%%%%%%%%%%%%%%%
\begin{lemma}[{\textsl{a~priori} convergence; see, e.g., ~\cite[Lemma~13]{bprr18++}}]\label{lemma:apriori}
Let $V$ be a Hilbert space. Let $a: V \times V \to \R$ be an elliptic and continuous bilinear form.
Let $F \in V^*$ be a bounded linear functional.
For each $\ell \in \N_0$, let $V_\ell \subseteq V$ be a closed subspace such that $V_\ell \subseteq V_{\ell+1}$.
Furthermore, define the limiting space $V_\infty := \overline{\bigcup_{\ell = 0}^\infty V_\ell} \subseteq V$.
Then, for all $\ell \in \N_0 \cup \{\infty\}$, there exists a unique Galerkin solution $u_\ell \in V_\ell$ satisfying
\begin{equation*} %\label{eq1:lemma:apriori}
 a(u_\ell, v_\ell) = F(v_\ell) 
 \quad \text{for all } v_\ell \in V_\ell.
\end{equation*}
Moreover, there holds
$\lim\limits_{\ell \to \infty} \norm{u_\infty - u_\ell}{V} = 0$.
\qed
\end{lemma}
%%%%%%%%%%%%%%%%%%%%

We will exploit Lemma~\ref{lemma:apriori} for the limiting multilevel space
$\V_\infty = \overline{\bigcup_{\ell=0}^\infty \V_\ell}$
as well as for the limiting finite element spaces
$\X_{\infty\nu} := \overline{\bigcup_{\ell = 0}^\infty \X_{\ell\nu}}$, $\nu \in \III$,
which are well-defined with the understanding that $\X_{\ell\nu} = \{0\}$ for $\nu \in \III \setminus \PPP_\ell$.

The next proposition replicates Proposition~10 in~\cite{bprr18++};
it states that the parametric enrichment satisfying the D{\"o}rfler marking criterion along a subsequence
guarantees convergence of the whole sequence of parametric error estimates.
The proof is independent of the structure of the underlying finite element spaces and, therefore,
carries over from~\cite{bprr18++} without changes.

%%%%%%%%%%%%%%%%%%%%
\begin{proposition}\label{prop10}
Let 
$\rho_\ell(\nu) \in \Big\{ \mu_\ell(\nu),\, \big(\mu_\ell(\nu)^2 + \zeta_\ell(\nu)^2\big)^{1/2} \Big\}$
for each $\nu \in \PPP_\ell$ ($\ell \in \N_0$).
Let $0 < \vartheta \le 1$.
Suppose that Algorithm~\ref{algorithm}
yields a subsequence $(\ell_k)_{k \in \N_0}$ such that 
\begin{equation} \label{eq:doerfler:vartheta:parametric}
   \vartheta \sum_{\nu \in \QQQ_{\ell_k}} \rho_{\ell_k}(\nu)^2
   \le \sum_{\nu \in \MMM_{\ell_k}} \rho_{\ell_k}(\nu)^2.
\end{equation}
Then, there holds convergence
$\displaystyle \sum_{\nu \in \QQQ_{\ell}} \rho_{\ell}(\nu)^2 \xrightarrow{\ell \to \infty} 0$.
\qed
\end{proposition}
%%%%%%%%%%%%%%%%%%%%

To prove a convergence result for the spatial contributions of error estimates,
we will use the following notation:
For $\omega \subset D$, we define
\begin{equation*}
 B_\omega(\vv,\ww) := \int_{\Gamma} \int_\omega a_0 \nabla \vv\cdot\nabla\ww \, \dx \, \dpi(\y)
 + \sum_{m=1}^\infty \int_{\Gamma} \int_\omega y_m a_m \nabla \vv\cdot\nabla\ww \, \dx \, \dpi(\y)
 \text{ for } \vv, \ww \in \V.
\end{equation*}
Note that $B_\omega(\cdot,\cdot)$ is symmetric, bilinear, and positive semidefinite.
We denote by $\enorm{\vv}{\omega} := B_\omega(\vv,\vv)^{1/2}$ the corresponding seminorm.
The following lemma is an analogue of Lemma~16 in~\cite{bprr18++}.
Since the result is formulated for individual indices $\nu \in \PPP_\ell$, the proof carries over from~\cite{bprr18++}
without significant modifications.

%%%%%%%%%%%%%%%%%%%%
\begin{lemma} \label{lem:prop:spatial}
%Let $\ww \in \{ \uu, \zz \}$.
Let $\nu \in \PPP_\ell$, $\xi \in \NN_{\ell\nu}^+$ and denote by
$\omega_{\ell\nu}(\xi) := \bigcup \{T \in \TT_{\ell\nu} : \xi \in T \}$ the associated vertex patch.
Then, the following estimates hold:
\begin{subequations} \label{eq1:msv}
   \begin{equation} \label{eq1a:msv}
      \mu_\ell(\nu,\xi) \leq C \enorm{\uu - \uu_\ell}{\omega_{\ell\nu}(\xi)},
   \end{equation}
   \begin{equation} \label{eq1b:msv}
      \mu_\ell(\nu,\xi)^2 + \zeta_\ell(\nu,\xi)^2 \leq C
      \Big( \enorm{\uu - \uu_\ell}{\omega_{\ell\nu}(\xi)}^2 + \enorm{\zz[\uu_\ell] - \zz_\ell[\uu_\ell]}{\omega_{\ell\nu}(\xi)}^2 \Big).
   \end{equation}
\end{subequations}
Furthermore, let $\uu_\infty \in \V$ (resp., $\zz_\infty[\uu_k] \in \V$ for $k \in \N_0$) be the limit of
$(\uu_\ell)_{\ell \in \N_0}$ (resp.,  $(\zz_\ell[\uu_k])_{\ell \in \N_0}$) guaranteed by Lemma~\ref{lemma:apriori}.
If $\widehat\varphi_{\ell\nu,\xi} \in \X_{\infty\nu}$, then there hold
\begin{subequations} \label{eq2:msv}
   \begin{equation} \label{eq2a:msv}
      \tau_\ell(\uu|\nu,\xi) = \mu_\ell(\nu,\xi) \leq C \enorm{\uu_\infty - \uu_\ell}{\omega_{\ell\nu}(\xi)},
   \end{equation}
   \begin{equation} \label{eq2b:msv}
      \tau_\ell(\uu|\nu,\xi)^2 + \tau_\ell(\zz[\uu_k]|\nu,\xi)^2 \leq C
      \Big( \enorm{\uu_\infty - \uu_\ell}{\omega_{\ell\nu}(\xi)}^2 +
      \enorm{\zz_\infty[\uu_k] - \zz_\ell[\uu_k]}{\omega_{\ell\nu}(\xi)}^2 \Big).
   \end{equation}
\end{subequations}
The constant $C>0$ in~\eqref{eq1:msv} and~\eqref{eq2:msv} depends only on $a_0$ and $\tau$.
\qed
\end{lemma}
%%%%%%%%%%%%%%%%%%%%

While Lemma~\ref{lem:prop:spatial} holds for each index $\nu \in \PPP_\ell$, its application in the convergence proof for
spatial error estimates in the multilevel setting will require
the following elementary lemma, which formulates a generalized dominated convergence result for sequences.
\ifsisc
%----------- SISC version -------------
The proof is elementary and is left to the reader.
\else
%----------- arXiv version -------------
For convenience of the reader, we include a simple proof in Appendix~\ref{sec:appendix_a}.
\fi

%%%%%%%%%%%%%%%%%%%%
\begin{lemma}\label{lemma:lebesque}
Let $(\alpha_n)_{n \in \N}, (\beta_n)_{n \in \N} \subset \R$ with
$\sum_{n = 1}^\infty |\beta_n| < \infty$.
Let $C > 0$.
For $k \in \N_0$, let $(\alpha_n^{(k)})_{n \in \N}, (\beta_n^{(k)})_{n \in \N} \subset \R$
with $|\alpha_n^{(k)}| \le C \, |\beta_n^{(k)}|$ and $\alpha_n^{(k)} \to \alpha_n$ as $k \to \infty$, for all $n \in \N$.
Then, the convergence $\sum_{n = 1}^\infty |\beta_n - \beta_n^{(k)}| \to 0$ as $k \to \infty$
implies that $\sum_{n = 1}^\infty |\alpha_n| < \infty$ and
$\sum_{n = 1}^\infty |\alpha_n - \alpha_n^{(k)}| \to 0$ as $k \to \infty$.
\ifsisc
%----------- SISC version -------------
\qed
%----------- arXiv version -------------
%
\fi
\end{lemma}
%%%%%%%%%%%%%%%%%%%%

With Lemmas~\ref{lem:prop:spatial} and~\ref{lemma:lebesque} at hand,
we can extend the result
established in~\cite[Proposition~11]{bprr18++} for single-level SGFEM to the multilevel~setting.

%%%%%%%%%%%%%%%%%%%%
\begin{proposition}\label{prop:conv:spatial}
%Let $\ww \in \{ \uu , \zz\}$.
Let $0 < \vartheta \le 1$.
Let $\rho_\ell(\nu,\xi) \in \Big\{ \mu_\ell(\nu,\xi),\, \big(\mu_\ell(\nu,\xi)^2 + \zeta_\ell(\nu,\xi)^2\big)^{1/2} \Big\}$
for each $\nu \in \PPP_\ell$ and $\xi \in \NN_{\ell\nu}^+$ ($\ell \in \N_0$).
Suppose that Algorithm~\ref{algorithm}
yields a subsequence $(\ell_k)_{k \in \N_0}$ such that 
\begin{equation}\label{eq1:prop:conv:spatial}
 \vartheta \sum_{\nu \in \PPP_{\ell_k}} \sum_{\xi \in \NN_{\ell_k\nu}^+} \rho_{\ell_k}(\nu,\xi)^2
 \le \sum_{\nu \in \PPP_{\ell_k}} \sum_{\nu \in \MM_{\ell_k\nu}} \rho_{\ell_k}(\nu,\xi)^2.
\end{equation}
Then, there holds convergence
$\displaystyle \sum_{\nu \in \PPP_{\ell_k}} \sum_{\xi \in \NN_{\ell_k\nu}^+} \rho_{\ell_k}(\nu,\xi)^2 \xrightarrow{k \to \infty} 0$.
\end{proposition}
%%%%%%%%%%%%%%%%%%%%

%%%%%%%%%%%%%%%%%%%%
\begin{proof}
The proof follows the lines of our own work~\cite[Proposition~11]{bprr18++} and builds upon~\cite[Theorem~2.1]{msv08}.
Therefore, in the same way as it was done in the proof of Proposition~11 in~\cite{bprr18++},
we sketch the main arguments and highlight how the results of~\cite{msv08} for deterministic problems can be extended to the parametric setting.
While the observations and notation in Steps~1--2 are essentially the same as in~\cite{bprr18++},
we note that Steps~3--6 are considerably more involved because of the present multilevel structure.

{\bf Step~1.}
The variational problems~\eqref{eq:weakform} and~\eqref{eq:nonlinear:weakform:dual},
their discretizations,
and the proposed adaptive algorithm satisfy the general framework
described in~\cite[section~2]{msv08}:
\begin{itemize}
\item the variational problems~\eqref{eq:weakform} and~\eqref{eq:nonlinear:weakform:dual}
fit into the class of problems considered in~\cite[section~2.1]{msv08};
\item the Galerkin discretizations~\eqref{eq:discrete_formulation} and~\eqref{eq:nonlinear:weakform:dual:discrete}
satisfy the assumptions in~\cite[equations~(2.6)--(2.8)]{msv08};
\item the spatial NVB refinement considered in the present paper satisfies the assumptions
on the mesh refinement in~\cite[equations~(2.5) and~(2.14)]{msv08};
\item the D\"orfler marking criterion~\eqref{eq1:prop:conv:spatial} implies the weak marking condition in~\cite[equation~(2.13)]{msv08};
\item finally, Lemma~\ref{lem:prop:spatial} proves the local discrete efficiency estimate
in the parametric setting (cf.~\cite[equation~(2.9b)]{msv08}).
Note that the global reliability of the estimator (see the lower bound of~\eqref{eq2:thm:estimator}
and~\cite[equation~(2.9a)]{msv08})
is not exploited here (and hence, not needed for the proof of Theorem~\ref{thm:nonlinear:plain_convergence}).
In particular, the estimates~\eqref{eq1:msv} and~\eqref{eq2:msv} from Lemma~\ref{lem:prop:spatial}
replace~\cite[eq.~(2.9b)]{msv08} and~\cite[eq.~(4.11)]{msv08}, respectively.
\end{itemize}

{\bf Step~2.}
Let $\nu \in \PPP_\infty := \bigcup_{\ell = 0}^\infty \PPP_\ell$. 
Let $\TT_{\infty\nu} := \bigcup_{k \ge 0} \bigcap_{\ell \ge k} \TT_{\ell\nu}$ be the set of all elements which remain unrefined after finitely many steps of refinement, where $\TT_{\ell\nu} = \emptyset$ if $\nu \not\in \PPP_\ell$.
In the spirit of~\cite[eqs.~(4.10)]{msv08}, for all $\ell \in \N_0$, we consider the decomposition
$ \TT_{\ell\nu} = \TT_{\ell\nu}^{\rm good} \cup \TT_{\ell\nu}^{\rm bad} \cup  \TT_{\ell\nu}^{\rm neither}$,
where
\begin{align*}
 \TT_{\ell\nu}^{\rm good} &:= \{T \in \TT_{\ell\nu} :
 \widehat\varphi_{\ell\nu,\xi} \in \X_{\infty\nu} \text{ for all } \xi \in \NN_{\ell\nu}^+ \cap T \}, \\
 \TT_{\ell\nu}^{\rm bad} &:= \{T \in \TT_{\ell\nu} : T' \in \TT_{\infty\nu} \text{ for all } T' \in \TT_{\ell\nu} \text{ with } T \cap T' \neq \emptyset \}, \\
 \TT_{\ell\nu}^{\rm neither} &:= \TT_{\ell\nu} \setminus (\TT_{\ell\nu}^{\rm good} \cup \TT_{\ell\nu}^{\rm bad}).
\end{align*}
The elements in $\TT_{\ell\nu}^{\rm good}$ are refined sufficiently many times in order to guarantee~\eqref{eq2:msv}.
The set $\TT_{\ell\nu}^{\rm bad}$ consists of all elements such that the whole element patch remains unrefined.
The remaining elements are collected in the set $\TT_{\ell\nu}^{\rm neither}$.
Note that $\TT_{\ell\nu}^{\rm good}$ is slightly larger than the corresponding set $\GG_{\ell\nu}^0$ in~\cite[eq.~(4.10a)]{msv08},
while $\TT_{\ell\nu}^{\rm bad}$ coincides with the corresponding set $\GG_{\ell\nu}^+$ in~\cite[eq.~(4.10b)]{msv08}.
As a consequence, $\TT_{\ell\nu}^{\rm neither}$ is smaller than the corresponding set $\GG_{\ell\nu}^*$ in~\cite[eq.~(4.10c)]{msv08}.

{\bf Step~3.}
In this step, we consider the two cases of $\rho_\ell(\nu,\xi)$ separately.
Let $\rho_\ell(\nu,\xi) = \mu_\ell(\nu,\xi)$.
By arguing as in the proof of Proposition~4.1 in~\cite{msv08}, we exploit
the uniform shape-regularity of the mesh $\TT_{\ell\nu}$ guaranteed by NVB
and use Lemma~\ref{lem:prop:spatial} and Lemma~\ref{lemma:apriori} to prove that
\begin{align}
 \sum_{T \in \TT_{\ell\nu}^{\rm good}} \sum_{\xi \in \NN_{\ell\nu}^+ \cap T} \mu_\ell(\nu,\xi)^2
 & \stackrel{\eqref{eq2a:msv}}{\lesssim}
 \sum_{T \in \TT_{\ell\nu}^{\rm good}} \sum_{\xi \in \NN_{\ell\nu}^+ \cap T} \enorm{\uu_\infty - \uu_\ell}{\omega_{\ell\nu}(\xi)}^2
\nonumber\\
 & \stackrel{\phantom{\eqref{eq2b:msv}}}{\lesssim} \enorm{\uu_\infty - \uu_{\ell}}{}^2 
 \xrightarrow{\ell \to \infty} 0.
 \label{eq:msv:step1}
\end{align}

Let $D_{\ell\nu}^{\rm neither} := \bigcup\{T' \in \TT_{\ell\nu} : T\cap T' \neq \emptyset \text{ for some } T \in \TT_{\ell\nu}^{\rm neither} \}$.
Since $\TT_{\ell\nu}^{\rm neither}$ is contained in the corresponding set $\GG_{\ell\nu}^*$ in~\cite[eq.~(4.10c)]{msv08},
arguing as in Step~1 of the proof of Proposition~4.2 in~\cite{msv08},
we show that $|D_{\ell\nu}^{\rm neither}| \to 0$ as $\ell \to \infty$.
Hence, Lemma~\ref{lem:prop:spatial}, uniform shape regularity, and
the fact that the local energy seminorm is absolutely continuous with respect to the Lebesgue measure,
i.e., $\enorm{\vv}{\omega} \to 0$ as $\vert\omega\vert \to 0$ for all $\vv \in \V$,
lead to
\begin{align}
%\begin{split}
 \sum_{T \in \TT_{\ell\nu}^{\rm neither}} \sum_{\xi \in \NN_{\ell\nu}^+ \cap T} \mu_\ell(\nu,\xi)^2
  & \stackrel{\eqref{eq1a:msv}}{\lesssim}
  \sum_{T \in \TT_{\ell\nu}^{\rm neither}} \sum_{\xi \in \NN_{\ell\nu}^+ \cap T} \enorm{\uu - \uu_\ell}{\omega_{\ell\nu}(\xi)}^2
  \nonumber
  \\
  & \stackrel{\phantom{\eqref{eq1b:msv}}}{\lesssim} \enorm{\uu - \uu_\ell}{D_{\ell\nu}^{\rm neither}}^2
  \xrightarrow{\ell \to \infty} 0.
 \label{eq:msv:step2}
%\end{split}
\end{align}
Now, let
$\rho_\ell(\nu,\xi) = \big(\mu_\ell(\nu,\xi)^2 + \zeta_\ell(\nu,\xi)^2\big)^{1/2}
  \reff{eq:nonlinear:tau:abbreviations:local}= \big(\tau_\ell(\uu|\nu,\xi)^2 + \tau_\ell(\zz[\uu_\ell]|\nu,\xi)^2\big)^{1/2}$.
To verify the analogue of~\eqref{eq:msv:step1} in this case,
note that
\begin{align*}
\sum_{T \in \TT_{\ell\nu}^{\rm good}} \sum_{\xi \in \NN_{\ell\nu}^+ \cap T} 
\big[ \tau_\ell(\uu|\nu,\xi)^2 + \tau_\ell(\zz[\uu_\ell]|\nu,\xi)^2 \big]
\reff{eq2b:msv}\lesssim
\bnorm{\uu_\infty - \uu_\ell}^2 + \bnorm{\zz_\infty[\uu_\ell] - \zz_\ell[\uu_\ell]}^2.
\end{align*}
We also note the \textsl{a~priori} convergence result
$\bnorm{\zz_\infty[\uu_\infty] - \zz_\ell[\uu_\infty]} + \bnorm{\zz_\ell[\uu_\infty] - \zz_\ell[\uu_\ell]} \to 0$
as $\ell \to \infty$,
where the limiting functions $\uu_\infty, \zz_\infty[\uu_\infty] \in \V$ are provided by Lemma~\ref{lemma:apriori}.
Therefore, the triangle inequality and Lemma~\ref{prop:nonlinear:goal:aux:stability} prove that
\begin{align*}
 \bnorm{\zz_\infty[\uu_\ell] - \zz_\ell[\uu_\ell]}
 &\le \bnorm{\zz_\infty[\uu_\ell] - \zz_\infty[\uu_\infty]}
 + \bnorm{\zz_\infty[\uu_\infty] - \zz_\ell[\uu_\infty]}
 + \bnorm{\zz_\ell[\uu_\infty] - \zz_\ell[\uu_\ell]}
 \\
 &\lesssim 
 \bnorm{\uu_\infty - \uu_\ell} 
 + \bnorm{\zz_\infty[\uu_\infty] - \zz_\ell[\uu_\infty]}
 \xrightarrow{\ell \to \infty} 0.
\end{align*}
Hence, we are led to
\begin{align*}
 \sum_{T \in \TT_{\ell\nu}^{\rm good}} \sum_{\xi \in \NN_{\ell\nu}^+ \cap T} 
 \big[ \tau_\ell(\uu|\nu,\xi)^2 + \tau_\ell(\zz[\uu_\ell]|\nu,\xi)^2 \big]
 \xrightarrow{\ell \to \infty} 0.
\end{align*}
Similar observations verify the analogue of~\eqref{eq:msv:step2}.
Indeed,
\begin{align*}
 \sum_{T \in \TT_{\ell\nu}^{\rm neither}} \sum_{\xi \in \NN_{\ell\nu}^+ \cap T} 
 \!\big[ \tau_\ell(\uu|\nu,\xi)^2 \,{+}\, \tau_\ell(\zz[\uu_\ell]|\nu,\xi)^2 \big]
 \!\!\reff{eq1b:msv}\lesssim\!\!
 \enorm{\uu \,{-}\, \uu_\ell}{D_{\ell\nu}^{\rm neither}}^2
 \,{+}\, \enorm{\zz[\uu_\ell] \,{-}\, \zz_\ell[\uu_\ell]}{D_{\ell\nu}^{\rm neither}}^2.
\end{align*}
Since $\bnorm{\zz[\uu_\ell] - \zz_\ell[\uu_\ell]} \to \bnorm{\zz[\uu_\infty] - \zz_\infty[\uu_\infty]}$
and $|D_{\ell\nu}^{\rm neither}| \to 0$ as $\ell \to \infty$, we have
\begin{align*}
 \sum_{T \in \TT_{\ell\nu}^{\rm neither}} \sum_{\xi \in \NN_{\ell\nu}^+ \cap T} 
 \big[ \tau_\ell(\uu|\nu,\xi)^2 + \tau_\ell(\zz[\uu_\ell]|\nu,\xi)^2 \big]
 \xrightarrow{\ell \to \infty} 0.
\end{align*}
Thus, for both cases of $\rho_\ell(\nu,\xi)$, we have proved that
\begin{align} \label{eq:msv:step12:rho}
 \sum_{T \in \TT_{\ell\nu}^{\rm good}} \sum_{\xi \in \NN_{\ell\nu}^+ \cap T}
 \rho_\ell(\nu,\xi)^2
 +
 \sum_{T \in \TT_{\ell\nu}^{\rm neither}} \sum_{\xi \in \NN_{\ell\nu}^+ \cap T} 
 \rho_\ell(\nu,\xi)^2 
 \xrightarrow{\ell \to \infty} 0.
\end{align}

{\bf Step~4.}
The aim of this step is to strengthen~\eqref{eq:msv:step12:rho} so that the convergence holds
for the sum over multi-indices $\nu \in \PPP_\ell$.
We will show this for $\rho_\ell(\nu,\xi) = \mu_\ell(\nu,\xi)$, with all the arguments applying
to the case of $\rho_\ell(\nu,\xi) = \big(\mu_\ell(\nu,\xi)^2 + \zeta_\ell(\nu,\xi)^2\big)^{1/2}$ without changes.

Recall that the index set $\III$ is countable so that we can identify each index $\nu \in \III$ with
a natural number $n \in \N$. For $\ell \in \N_0$, we consider the following~sequence:
\[
(\alpha_n^{(\ell)})_{n \in \N}
 = (\alpha_\nu^{(\ell)})_{\nu \in \III}
 := \bigg( \sum_{T \in \TT_{\ell\nu}^{\rm good}} \sum_{\xi \in \NN_\ell^+ \cap T} \mu_\ell(\nu,\xi)^2
                +
                \sum_{T \in \TT_{\ell\nu}^{\rm neither}} \sum_{\xi \in \NN_{\ell\nu}^+ \cap T} \mu_\ell(\nu,\xi)^2
      \bigg)_{\nu \in \III},
\]
where $\TT_{\ell\nu} = \emptyset$ and, consequently, $\alpha_\nu^{(\ell)} = 0$ if $\nu \in \III \backslash \PPP_\ell$.
We already know from~\eqref{eq:msv:step12:rho}
that $\alpha_\nu^{(\ell)} \to 0 =: \alpha_\nu$ as $\ell \to \infty$, for all $\nu \in \III$.
Arguing as in the proof of~\cite[Lemma~5, Step~2]{bpr2020+}, we find that
\begin{align*}
 0 \,{\le}\, \alpha_\nu^{(\ell)} 
 \,{\le}\, \sum_{T \in \TT_{\ell\nu}} \sum_{\xi \in \NN_{\ell\nu}^+ \cap T} \mu_\ell(\nu,\xi)^2
 \,{ \lesssim}\, \sum_{\xi \in \NN_{\ell\nu}^+} \mu_\ell(\nu,\xi)^2
 \,{\lesssim}\, \norm{\widehat e_{\ell\nu}}{D}^2 \,{=:}\, \beta_\nu^{(\ell)},
\end{align*}
where $\widehat e_{\ell\nu} \in \widehat\X_{\ell\nu}$ solves
\begin{equation} \label{eq:lem:prop:spatial:b}
 \dual{\widehat e_{\ell\nu}}{\widehat v_{\ell\nu}} = B(\widehat\uu_\ell - \uu_\ell, \widehat v_{\ell\nu} P_\nu)
 \quad \text{for all } \widehat v_{\ell\nu} \in \widehat\X_{\ell\nu}.
\end{equation}
Defining
$\V_\ell'' := \bigoplus_{\nu \in \PPP_\ell} \big[ \widehat\X_{\ell\nu} \otimes {\rm span}\{P_\nu\} \big]
\subseteq \widehat\V_\ell$
and $\ee''_\ell := \sum_{\nu \in \PPP_\ell} \widehat e_{\ell\nu} P_\nu$,
we conclude from~\eqref{eq:lem:prop:spatial:b} and~\eqref{eq:galerkin:aux2} with $\ww = \uu$
that $\ee''_\ell \in \V_\ell''$ is the unique solution to
\begin{align*}
  B_0(\ee''_\ell, \vv''_\ell) = B(\uu - \uu_\ell, \vv''_\ell)
 \quad \text{for all } \vv''_\ell \in \V_\ell''.
\end{align*}
Since $\V_\ell'' \subseteq \V_{\ell+1}''$ for all $\ell \in \N_0$ and since $\ww_\ell \to \ww_\infty$ as $\ell \to \infty$,
we can argue as in the proof of~\cite[Lemma~14]{bprr18++}
to see that Lemma~\ref{lemma:apriori} provides $\ee_\infty'' = \sum_{\nu \in \III} \widehat e_{\infty\nu} P_\nu \in \V$ such that
\begin{align*}
 \sum_{\nu \in \III} \norm{\widehat e_{\infty\nu} - \widehat e_{\ell\nu}}{D}^2
 = \enorm{\ee''_\infty - \ee''_\ell}{0}^2 \xrightarrow{\ell \to \infty} 0,
\end{align*}
where $\widehat e_{\ell\nu} = 0$ if $\nu \in \III \backslash \PPP_\ell$.
In particular, it follows that 
\begin{align*}
 \sum_{\nu \in \III} \big| \norm{\widehat e_{\infty\nu}}{D}^2 - \norm{\widehat e_{\ell\nu}}{D}^2 \big|
 &= 
 \sum_{\nu \in \III} \big| \norm{\widehat e_{\infty\nu}}{D} - \norm{\widehat e_{\ell\nu}}{D} \big|
 \big[ \norm{\widehat e_{\infty\nu}}{D} + \norm{\widehat e_{\ell\nu}}{D} \big]
 \\&
 \le
 \sum_{\nu \in \III} \norm{\widehat e_{\infty\nu} - \widehat e_{\ell\nu}}{D} \big[ \norm{\widehat e_{\infty\nu}}{D} + \norm{\widehat e_{\ell\nu}}{D} \big]
 \\& 
 \le 
 2 \big[ \enorm{\ee''_\infty}{0} + \enorm{\ee''_\ell}{0} \big] 
 \, \enorm{\ee''_\infty - \ee''_\ell}{0}
 \xrightarrow{\ell \to \infty} 0.
\end{align*}
With $\beta_\nu := \norm{\widehat e_{\infty\nu}}{D}^2$, we can thus apply Lemma~\ref{lemma:lebesque}
to strengthen the parameter-wise convergence to
\begin{align*}
 \sum_{\nu \in \III} \alpha_\nu^{(\ell)}  = \sum_{\nu \in \PPP_\ell} \alpha_\nu^{(\ell)} \xrightarrow{\ell \to \infty} 0.
\end{align*}
In explicit terms, this proves that for both cases of $\rho_\ell(\nu,\xi)$, the convergence result in~\eqref{eq:msv:step12:rho}
can indeed be strengthened to 
\begin{align}\label{eq:msv:step1-2}
 \sum_{\nu \in \PPP_\ell}
 \bigg( \sum_{T \in \TT_{\ell\nu}^{\rm good}} \sum_{\xi \in \NN_{\ell\nu}^+ \cap T} \rho_\ell(\nu,\xi)^2
 + \sum_{T \in \TT_{\ell\nu}^{\rm neither}} \sum_{\xi \in \NN_{\ell\nu}^+ \cap T} \rho_\ell(\nu,\xi)^2
 \bigg)
 \xrightarrow{\ell \to \infty} 0.
\end{align}

{\bf Step~5.}
To conclude the proof, it remains to consider the sets $\TT_{\ell\nu}^{\rm bad}$.
If $\xi \in \MM_{\ell_k\nu}$ and $T \in \TT_{\ell_k\nu}$ with $\xi \in T$, then
$T \in \TT_{\ell_k\nu} \setminus \TT_{\ell_k\nu}^{\rm bad} = \TT_{\ell_k\nu}^{\rm good} \cup \TT_{\ell_k\nu}^{\rm neither}$.
Therefore, it follows from~\eqref{eq:msv:step1-2} that
\begin{equation*}
 \vartheta \sum_{\nu \in \PPP_{\ell_k}} \sum_{\xi \in \NN_{\ell_k\nu}^+} \rho_{\ell_k}(\nu,\xi)^2
 \reff{eq1:prop:conv:spatial}\le
 \sum_{\nu \in \PPP_{\ell_k}} \sum_{\xi \in \MM_{\ell_k\nu}} \rho_{\ell_k}(\nu,\xi)^2
 \xrightarrow{k \to \infty} 0.
\end{equation*}
In particular, we obtain (cf.~\cite[eq.~(4.17)]{msv08})
\begin{equation} \label{eq:msv:step2:1}
 \sum_{\xi \in \NN_{\ell_k\nu}^+ \cap T} \rho_{\ell_k}(\nu,\xi)^2
 \xrightarrow{k \to \infty} 0
 \quad \text{for all } \nu \in \PPP_\ell \text{ and all } T \in \TT_{\ell_k \nu}^{\rm bad}.
\end{equation}
Arguing as in Steps~2--5 of the proof of Proposition~4.3 in~\cite{msv08},
we use~\eqref{eq:msv:step2:1} and
apply the Lebesgue dominated convergence theorem to derive that
\begin{equation*}
 \sum_{T \in \TT_{\ell_k \nu}^{\rm bad}} \sum_{\xi \in \NN_{\ell_k \nu}^+ \cap T} \tau_{\ell_k}(\ww|\nu,\xi)^2
 \xrightarrow{k \to \infty} 0
 \quad \text{for all } \nu \in \PPP_\ell.
\end{equation*}
As in Step~4, this parameter-wise convergence can be strengthened to  
\begin{align}\label{eq:msv:step3}
 \sum_{\nu \in \PPP_{\ell_k}} \sum_{T \in \TT_{\ell_k \nu}^{\rm bad}}
 \sum_{\xi \in \NN_{\ell_k \nu}^+ \cap T} \rho_{\ell_k}(\nu,\xi)^2
 \xrightarrow{k \to \infty} 0.
\end{align}

{\bf Step~6.}
Combining~\eqref{eq:msv:step1-2} and~\eqref{eq:msv:step3}, we obtain
\begin{align*}
 \sum_{\nu \in \PPP_{\ell_k}} \sum_{\xi \in \NN_{\ell_k\nu}^+} \rho_{\ell_k}(\nu,\xi)^2
 \,{\le}
 \sum_{\nu \in \PPP_{\ell_k}} \bigg(
 &
 \sum_{T \in \TT_{\ell_k\nu}^{\rm good}} \sum_{\xi \in \NN_{\ell_k\nu}^+ \cap T} \rho_{\ell_k}(\nu,\xi)^2
 +
 \sum_{T \in \TT_{\ell_k \nu}^{\rm bad}} \sum_{\xi \in \NN_{\ell_k \nu}^+ \cap T} \rho_{\ell_k}(\nu,\xi)^2
 \\& \quad
 +
 \sum_{T \in \TT_{\ell_k\nu}^{\rm neither}} \sum_{\xi \in \NN_{\ell_k\nu}^+ \cap T} \rho_{\ell_k}(\nu,\xi)^2
 \bigg)  \xrightarrow{k \to \infty} 0.
\end{align*}
This concludes the proof.
\end{proof}
%%%%%%%%%%%%%%%%%%%%

We are now in a position to prove our main result.

\begin{proof}[Proof of Theorem~\ref{thm:nonlinear:plain_convergence}]
The proof is split into five steps.

{\bf Step~1.}
Let $(\ell_k')_{k \in \N_0}$ be the sequence of iterations,
where the marking strategy of Algorithm~\ref{algorithm} selects $\MMM_{\ell_k} = \MMM_{\ell_k}'$ and $\MM_{\ell_k\nu} = \MM_{\ell_k\nu}'$
for all $\nu \in \PPP_{\ell_k}$ (i.e., marking with respect to the primal error estimate $\mu_\ell$).
Let $(\ell_k'')_{k \in \N_0}$ be the index sequence, where the marking strategy of Algorithm~\ref{algorithm} selects
$\MMM_{\ell_k} = \MMM_{\ell_k}''$ and $\MM_{\ell_k\nu} = \MM_{\ell_k\nu}''$ for all $\nu \in \PPP_{\ell_k}$
(i.e., marking with respect to the combined primal-dual error estimate $\big[ \mu_\ell^2 + \zeta_\ell^2 \big]^{1/2}$).
Note that this provides a partitioning of the sequence
$\Big( \mu_\ell \, \big[ \mu_\ell^2 + \zeta_\ell^2 \big]^{1/2} \Big)_{\ell \in \N_0}$
into two disjoint subsequences
$\Big( \mu_{\ell_k'} \, \big[ \mu_{\ell_k'}^2 + \zeta_{\ell_k'}^2 \big]^{1/2} \Big)_{k \in \N_0}$ and
$\Big( \mu_{\ell_k''} \, \big[ \mu_{\ell_k''}^2 + \zeta_{\ell_k''}^2 \big]^{1/2} \Big)_{k \in \N_0}$.
Without loss of generality (as the following arguments will show), we can assume that both subsequences are countably infinite.

{\bf Step~2.} 
In this step, we show the convergence $\mu_{\ell_k'} \to 0$
along the iteration sequence $(\ell_k')_{k \in \N_0}$, where the primal error estimate is employed for marking, i.e.,
\begin{align*}
 \theta \mu_{\ell_k'}^2 
 \le \sum_{\nu \in \PPP_{\ell_k'}} \sum_{\xi \in \MM'_{\ell_k'\nu}}  \mu_{\ell_k'}(\nu,\xi)^2 + \sum_{\nu \in \MMM'_{\ell_k'}} \mu_{\ell_k'}(\nu)^2.
\end{align*}
To this end, the sequence $(\ell_k')_{k \in \N_0}$ is further partitioned into
two disjoint subsequences $(\ell_k^{\prime+})_{k \in \N_0}$ and $(\ell_k^{\prime-})_{k \in \N_0}$, where
\begin{itemize}
\item $\displaystyle \sum_{\nu \in \PPP_{\ell_k^{\prime +}}} \sum_{\xi \in \MM'_{\ell_k^{\prime+}\nu}}  \mu_{\ell_k^{\prime+}}(\nu,\xi)^2 \ge \frac{1}{2} \, \bigg( \sum_{\nu \in \PPP_{\ell_k^{\prime+}}} \sum_{\xi \in \MM'_{\ell_k^{\prime+}\nu}}  \mu_{\ell_k^{\prime+}}(\nu,\xi)^2 + \sum_{\nu \in \MMM'_{\ell_k^{\prime+}}} \mu_{\ell_k^{\prime+}}(\nu)^2 \bigg)$,
\item $\displaystyle \sum_{\nu \in \PPP_{\ell_k^{\prime-}}} \sum_{\xi \in \MM'_{\ell_k^{\prime-}\nu}}  \mu_{\ell_k^{\prime-}}(\nu,\xi)^2 < \frac{1}{2} \, \bigg( \sum_{\nu \in \PPP_{\ell_k^{\prime-}}} \sum_{\xi \in \MM'_{\ell_k^{\prime-}\nu}}  \mu_{\ell_k^{\prime-}}(\nu,\xi)^2 + \sum_{\nu \in \MMM'_{\ell_k^{\prime-}}} \mu_{\ell_k^{\prime-}}(\nu)^2 \bigg)$,
\end{itemize}
respectively.
Again, without loss of generality (as the following arguments will show),
we assume that also these two subsequences are countably infinite.

{\bf Step~2a.} 
Along the sequence $(\ell_k^{\prime+})_{k \in \N_0}$, by definition, it follows that
\begin{align*}
\theta \sum_{\nu \in \PPP_{\ell_k^{\prime+}}} \sum_{\xi \in \NN_{\ell_k^{\prime+}\nu}^+}  \mu_{\ell_k^{\prime+}}(\nu,\xi)^2
\le \theta \mu_{\ell_k^{\prime+}}^2
\le 2 \sum_{\nu \in \PPP_{\ell_k^{\prime+}}} \sum_{\xi \in \MM'_{\ell_k^{\prime+}\nu}}
\mu_{\ell_k^{\prime+}}(\nu,\xi)^2,
\end{align*}
i.e., there holds the D\"orfler marking criterion~\eqref{eq1:prop:conv:spatial}
for spatial discretizations with $\vartheta = \theta/2$.
Therefore, Proposition~\ref{prop:conv:spatial} proves that 
\begin{equation*}
 \sum_{\nu \in \PPP_{\ell_k^{\prime+}}} \sum_{\xi \in \NN_{\ell_k^{\prime+}\nu}^+} \mu_{\ell_k^{\prime+}}(\nu,\xi)^2 \xrightarrow{k \to \infty} 0
\end{equation*}
and, hence, also $\mu_{\ell_k^{\prime+}}^2 \to 0$ as $k \to \infty$.

{\bf Step~2b.} 
Along the sequence $(\ell_k^{\prime-})_{k \in \N_0}$, by definition, it follows that
\begin{align*}
 \sum_{\nu \in \MMM'_{\ell_k^{\prime-}}} \mu_{\ell_k^{\prime-}}(\nu)^2  
 > \frac{1}{2} \, \bigg( \sum_{\nu \in \PPP_{\ell_k^{\prime-}}} \sum_{\xi \in \MM'_{\ell_k^{\prime-}\nu}}  \mu_{\ell_k^{\prime-}}(\nu,\xi)^2 + \sum_{\nu \in \MMM'_{\ell_k^{\prime-}}} \mu_{\ell_k^{\prime-}}(\nu)^2 \bigg)
\end{align*}
and, hence,
\begin{align*}
 \theta \sum_{\nu \in \QQQ_{\ell_k^{\prime-}}} \mu_{\ell_k^{\prime-}}(\nu)^2 
 \le \theta \mu_{\ell_k^{\prime-}}^2  
 < 2 \sum_{\nu \in \MMM'_{\ell_k^{\prime-}}} \mu_{\ell_k^{\prime-}}(\nu)^2, 
 \end{align*}
i.e., there holds the D\"orfler marking criterion~\eqref{eq:doerfler:vartheta:parametric}
for parametric discretizations with $\vartheta = \theta/2$. Therefore, Proposition~\ref{prop10} implies that
\begin{align*}
 \sum_{\nu \in \QQQ_{\ell_k^{\prime-}}} \mu_{\ell_k^{\prime-}}(\nu)^2 
 \xrightarrow{k \to \infty} 0
\end{align*}
and, hence, also $\mu_{\ell_k^{\prime-}}^2 \to 0$ as $k \to \infty$.

{\bf Step~2c.}
From the preceding Steps~2a--2b, we prove that the sequence $(\mu_{\ell_k'})_{k \in \N_0}$ can be partitioned into
two subsequences $(\mu_{\ell_k^{\prime+}})_{k \in \N_0}$ and $(\mu_{\ell_k^{\prime-}})_{k \in \N_0}$, which both converge to zero.
According to basic calculus, this implies that $\mu_{\ell_k'} \to 0$ as $k \to \infty$. 

{\bf Step~3.}
Note that the dual error estimate $\zeta_\ell$ defined in~\eqref{eq:nonlinear:tau:abbreviations}
is uniformly bounded,~as
\begin{align*}
   \zeta_\ell \reff{eq1:thm:estimator}\lesssim
   \bnorm{\widehat \zz_\ell[\uu_\ell] - \zz_\ell[\uu_\ell]}
   & \reff{eq:saturation_aux}\le
   \bnorm{\zz[\uu_\ell] - \zz_\ell[\uu_\ell]} \reff{eq:best_app}\le
   \bnorm{\zz[\uu_\ell]} \le
   \bnorm{\zz[\uu] - \zz[\uu_\ell]} + \bnorm{\zz[\uu]}
   \\
   & \reff{eq:nonlinear:goal:aux:stability}\le
   \Cgoal \, \bnorm{\uu - \uu_\ell} + \bnorm{\zz[\uu]} \reff{eq:best_app}\lesssim
   \bnorm{\uu} + \bnorm{\zz[\uu]}
   \quad \text{for all } \ell \in \N_0.
\end{align*}
Consequently, it follows from Step~2 that 
\begin{align}
\mu_{\ell_k'} \, \big[ \mu_{\ell_k'}^2 + \zeta_{\ell_k'}^2 \big]^{1/2} \xrightarrow{k \to \infty} 0.
\end{align}

{\bf Step~4.}
Note that the roles of the primal and the combined primal-dual error estimates
in all the preceding arguments in Steps~2--3 can be swapped.
Hence, it follows that
\begin{align}
 \mu_{\ell_k''} \, \big[ \mu_{\ell_k''}^2 + \zeta_{\ell_k''}^2 \big]^{1/2} \xrightarrow{k \to \infty} 0,
\end{align}
where we recall that the combined primal-dual error estimate
is employed for marking along the iteration sequence $(\ell_k'')_{k \in \N_0}$.

{\bf Step~5.}
Overall, we obtain that the sequence
$\Big( \mu_\ell \big[ \mu_\ell^2 {+} \zeta_\ell^2 \big]^{\!1/2} \Big)_{\!\ell \in \N_0}$
can be partitioned~into two subsequences
$\Big( \mu_{\ell_k'} \, \big[ \mu_{\ell_k'}^2 + \zeta_{\ell_k'}^2 \big]^{1/2} \Big)_{k \in \N_0}$ and
$\Big( \mu_{\ell_k''} \, \big[ \mu_{\ell_k''}^2 + \zeta_{\ell_k''}^2 \big]^{1/2} \Big)_{k \in \N_0}$, which both converge to zero.
According to basic calculus, this implies that
$\mu_\ell \, \big[ \mu_\ell^2 + \zeta_\ell^2 \big]^{1/2} \to 0$ as $\ell \to \infty$.
\end{proof}
%%%%%%%%%%%%%%%%%%%%

%%%%%%%%%%%%%%%%%%%%
\begin{remark} \label{rem:afem}
Note that standard adaptive SGFEM formally corresponds to the case, where the iteration sequence $(\ell_k'')$
in the proof of Theorem~\ref{thm:nonlinear:plain_convergence} is void.
Therefore, the proof of Theorem~\ref{thm:nonlinear:plain_convergence} also establishes plain convergence of the 
adaptive multilevel SGFEM algorithms from~\cite{bpr2020+}.
In particular, the analysis of the present work for combined D\"orfler marking (as employed in Algorithm~\ref{algorithm})
can be used to prove plain convergence of adaptive algorithms with
separate D\"orfler marking of spatial and parametric indicators (as done, e.g., in~\cite{bprr18++}).
\end{remark}
%%%%%%%%%%%%%%%%%%%%

%%%%%%%%%%%%%%%%%%%%
\section{Numerical results}  \label{sec:numerics}
%%%%%%%%%%%%%%%%%%%%

In this section, 
to illustrate the performance of Algorithm~\ref{algorithm}
and to underpin our theoretical findings,
we present a collection of numerical experiments in~2D.
All computations have been performed using the MATLAB toolbox Stochastic T-IFISS~\cite{brs2020,BespalovR_stoch_tifiss}.
Throughout this section,
we consider the parametric model problem~\eqref{eq:strongform} introduced in section~\ref{sec:problem} and assume
that each parameter in $\y = (y_m)_{m \in \N} \in \Gamma$ is the image of a uniformly distributed independent mean-zero random
variable, so that $\d{\pi_m}(y_m) = \d y_m/2$.

%%%%%%%%%%%%%%%%%%%%
%\subsection{Benchmark test problem} \label{sec:exp1}
%%%%%%%%%%%%%%%%%%%%

We consider four different setups for model problem~\eqref{eq:strongform} by
varying the physical domain $D \subset \R^2$ (square, L-shaped, and slit domains),
the right-hand side function $\ff$, and the goal functional $\gbold$.
The diffusion coefficient \rev{has the same representation} for all four setups.
\rev{It is the representation} introduced in~\cite[Section~11.1]{egsz14}
(and considered in many other works, e.g.,
\cite{egsz15,bs16,em16,br18,bprr18+,bpr2020+},
thus being a benchmark problem for testing novel discretization strategies).
Specifically, for every $x = (x_1,x_2) \in D$, we set $a_0(x) := 1$ and choose the coefficients $a_m(x)$ 
in~\eqref{eq1:a} to represent planar Fourier modes of increasing total order,~i.e.,
\begin{equation} \label{eq:Eigel:coefficient}
a_m(x) := A m^{-\rev{\sigma}} \cos(2\pi\beta_1(m) \, x_1) \cos(2\pi\beta_2(m)\, x_2)
\quad \text{for all $m \in \N$},
\end{equation}
\rev{where $\sigma>1$ determines the decay rate of the coefficient amplitudes},
$0 < A < 1/\zeta(\rev{\sigma})$ (here $\zeta(\cdot)$ denotes the Riemann zeta function),
$\beta_1$ and $\beta_2$ are defined as
$\beta_1(m) := m - k(m)(k(m) + 1)/2$
and
$\beta_2(m) := k(m) - \beta_1(m)$, respectively,
with $k(m) := \lfloor -1/2  + \sqrt{1/4+2m}\rfloor$ for all $m \in \N$. 
Assumption~\eqref{eq2:a} is satisfied with $a_0^{\rm min} = a_0^{\rm max} = 1$.
Since $\tau = A \, \zeta(\rev{\sigma})$, assumption~\eqref{eq3:a} is fulfilled for any choice of  $0 < A < 1/\zeta(\rev{\sigma})$.
We \rev{consider different values for $\sigma$ (specified below) and} set $A = 0.9 / \zeta(\rev{\sigma})$, which yields $\tau = 0.9$.

In each setup, the goal functional features a weight function $w \in L^\infty(D)$
that we use to introduce local spatial features in the corresponding QoI.
The adaptive algorithm terminates when the goal-oriented error estimate
$\mu_\ell (\mu_\ell^2 + \zeta_\ell^2)^{1/2}$ is smaller than a tolerance $\tol>0$.
Let $L \in \N_0$ denote the iteration at which the adaptive algorithm stops.
At each iteration $\ell \in \{0,1,\ldots,L\}$, the error in the goal functional, $\lvert \gbold(\uu) - \gbold(\uu_\ell) \rvert$,
is estimated by replacing the unknown exact solution $\uu \in \V$ to the primal problem~\eqref{eq:weakform}
by an accurate reference solution $\uu_\mathrm{ref} \in \V_{\mathrm{ref}}$.
Here, we set $\V_{\mathrm{ref}} := \SS^2_0(\TT_{\mathrm{ref}}) \otimes \{ P_\nu : \nu\in\PPP_{\mathrm{ref}} \}$,
i.e., $\V_{\mathrm{ref}}$ follows a single-level construction that employs piecewise quadratic (P2) finite element approximations
over a fine mesh $\TT_{\mathrm{ref}}$ and a large index set $\PPP_{\mathrm{ref}}$. More specifically, 
\begin{itemize}
\item the mesh $\TT_{\mathrm{ref}} := \widehat\TT_{L,\0}$ is the uniform refinement of the final mesh associated with the zero index
(for each setup, $\TT_{L,\0}$ is the finest mesh generated by the adaptive algorithm);
\item the index set $\PPP_{\mathrm{ref}} := \PPP_L \cup \MMM_L$ is the union of the final index set and the set of marked indices
at the final iteration.
\end{itemize}
For the sake of clarity and reproducibility,
in Table~\ref{tab:results} we show the stopping tolerance and the resulting values of $L$, \revx{$M_{\PPP_L}$ (the number of active parameters in the final SGFEM approximation $\uu_L$)}, $\dim \V_L$, and $\dim \V_\mathrm{ref}$ for each setup.

%%%%%%%%%%%%%%%%%%%%
\begin{table}[h]
\begin{tabular}{c|c|c|c|c|c|c|r}
& \rev{$\sigma$} & \rev{$A$} & $\tol$ & $L$ & $\revx{M_{\PPP_L}}$ & $\dim \V_L$ & $\dim \V_\mathrm{ref}$ \\
 \hline
{\Large \strut}%
Setup 1 & \rev{3/2} & \rev{0.345} & \rev{\num{7e-7}} & \rev{15} & \revx{15} & \rev{\num{192188}} & \rev{\num{457806195}} \\
Setup 2 & \rev{11/10} & \rev{0.085} & \num{5e-6} & \rev{15} & \revx{15} & \rev{\num{225251}} & \rev{\num{258115326}} \\
Setup 3 & \rev{4/3} & \rev{0.250} & \rev{\num{6e-7}} & \rev{16} & \revx{16} & \rev{\num{343317}} & \rev{\num{1088434319}} \\
Setup 4 & \rev{2} & \rev{0.547} & \num{6e-5} & 16 & \revx{16} & \num{290858} & \num{350589351} \\
\end{tabular}
\caption{
\rev{Parameter $\sigma$ and the resulting value of $A$ in~\eqref{eq:Eigel:coefficient},
the stopping tolerance $\tol$ and the resulting values of
$L$, $\revx{M_{\PPP_L}}$, $\dim \V_L$, and $\dim \V_\mathrm{ref}$ for all four setups.}}
\label{tab:results}
\end{table}
%%%%%%%%%%%%%%%%%%%%

\rev{As a further validation, for experimental setups where the goal functional
is represented in the form
\begin{equation*}
\gbold(\uu) = \int_\Gamma G(\uu(\cdot,\y)) \, \d{\pi}(\y)
\end{equation*}
for some $G: \X \to \R$,
we follow~\cite{egsz14,egsz15}
and compute a Monte Carlo-based approximation of the error in the goal functional
at each iteration of the adaptive algorithm.
To that end, we select a truncation parameter $M_{\aa} \in \N$ and denote by
$\big\{ \y^{(i)} = \big(y_1^{(i)},y_2^{(i)},\ldots,y_{M_{\aa}}^{(i)}\big) \big\}_{i=1}^{M_\mathrm{MC}}$
a set of
independent, uniformly distributed realizations of the random parameter vector
$(y_1,y_2,\ldots,y_{M_{\aa}}) \in [-1,1]^{M_{\aa}}$.
For each $i = 1,2,\ldots,M_\mathrm{MC}$,
we denote by $\widetilde u_h^{(i)}$ a finite element approximation
(computed on a sufficiently fine mesh) of the solution to
the following (deterministic) boundary value~problem:
\begin{equation*}
\begin{aligned}
 -\nabla \cdot (\widetilde\aa(x,\y^{(i)}) \nabla \widetilde u^{(i)}(x)) 
 &= \ff(x,\y^{(i)}), \quad && x \in D,\\
\widetilde u^{(i)}(x) & = 0, \quad && x \in \partial D,
\end{aligned}
\end{equation*}
where
\begin{equation} \label{eq:truncated:coeff}
\widetilde \aa(x,\y) = a_0(x) + \sum_{m=1}^{M_{\aa}} y_m a_m(x)
\quad \text{for } x \in D \text{ and } \y \in [-1,1]^{M_{\aa}}
\end{equation}
denotes an approximation of the diffusion coefficient $\aa(x,\y)$
obtained by truncating the expansion in~\eqref{eq1:a} after $M_{\aa}$ terms.
Then, the error in the goal functional can be approximated as follows:
\begin{equation*}
\begin{aligned}
\lvert \gbold(\uu) - \gbold(\uu_\ell) \rvert
&=
\bigg\vert \int_\Gamma [ G(\uu(\cdot,\y)) {-} G(\uu_\ell(\cdot,\y)) ] \, \d{\pi}(\y) \bigg\vert \\[3pt]
&\approx
\frac{1}{M_\mathrm{MC}} \sum_{i=1}^{M_\mathrm{MC}}
\big\lvert G(\widetilde u_h^{(i)}) - G(\uu_\ell(\cdot,\y^{(i)})) \big\rvert
=:
e_\ell^\mathrm{MC}.
\end{aligned}
\end{equation*}
\revx{Here, the number $M_\mathrm{MC}$ of Monte Carlo samples is chosen so that the empirical variance $\text{Var}_G$
of the vector $\big( \big\lvert G(\widetilde u_h^{(i)}) - G(\uu_\ell(\cdot,\y^{(i)})) \big\rvert \big)_{i=1}^{M_\mathrm{MC}}$ satisfies the inequality
$\sqrt{\text{Var}_G / M_\mathrm{MC}} \le \alpha \, \tol$ for some $\alpha \in (0,1)$.
Furthermore, the truncation parameter $M_{\aa}$ needs to be chosen so that the associated truncation error $\| \aa - \widetilde \aa \|_{L^\infty(D \times \Gamma)}$
is a fraction of $\sqrt{tol}$.
For diffusion coefficients represented by~\eqref{eq1:a}, the following coarse estimate of the truncation error can be obtained:
\begin{equation} \label{eq:truncation:error}
   \| \aa - \widetilde \aa \|_{L^\infty(D \times \Gamma)} \le \sum\limits_{m=M_{\aa}+1}^{\infty} \|a_m\|_{L^\infty(D)}.
\end{equation}
We note, however, that sharper results exist for diffusion coefficients represented by random fields
with piecewise analytic (resp., piecewise smooth) covariance functions,
where it is shown that the truncation error decays exponentially (resp., algebraically) with respect to $M_{\aa}$; see, e.g.,~\cite[Proposition~4.2]{fst05}).}
}

%%%%%%%%%%%%%%%%%%%%
\subsection{Problem specifications}
%%%%%%%%%%%%%%%%%%%%

Let us now describe the problem specifications for each setup.

%%%%%%%%%%%%%%%%%%%%
$\bullet$ \emph{Setup~1: expectation of a weighted $L^2$-norm}.
%%%%%%%%%%%%%%%%%%%%
The physical domain is the unit square $D = (0,1)^2$.
\rev{The decay rate of coefficient amplitudes in~\eqref{eq:Eigel:coefficient} is $\sigma=3/2$.}
The right-hand side function is constant: $\ff \equiv 1$ in $D$.
The goal functional is the expectation of the (squared) weighted $L^2$-norm:
\begin{equation}
\gbold(\uu) = \int_\Gamma \int_D w(x) \, \uu(x,\y)^2  \, \d{x} \, \d{\pi}(\y); \label{eq:quadratic_goal}
\end{equation}
cf.~\cite[Section~3.1]{bip2021}.
In this experiment, we choose $w = \chi_S / \vert S \vert$,
where $\chi_S : D \to \{ 0, 1\}$ is the characteristic function of the square $S = (5/8, 7/8) \times (9/16,13/16) \subset D$.
The initial mesh $\TT_0$ is a uniform mesh of 512 right-angled triangles.

%%%%%%%%%%%%%%%%%%%%
$\bullet$ \emph{Setup~2: expectation of a nonlinear convection term}.
%%%%%%%%%%%%%%%%%%%%
The physical domain is the L-shaped domain $D = (-1,1)^2 \setminus (-1,0]^2$.
\rev{We set $\sigma=11/10$ in~\eqref{eq:Eigel:coefficient}.}
The right-hand side function is constant: $\ff \equiv 1$ in $D$.
The goal functional is the expectation of a nonlinear convection term, i.e.,
\begin{equation*}
\gbold(\uu) = \int_\Gamma \int_D w(x) \, \uu(x,\y)
                 \bigg( \frac{\partial \uu}{\partial x_1}(x,\y) + \frac{\partial \uu}{\partial x_2}(x,\y) \bigg) \, \d{x} \, \d{\pi}(\y)
\end{equation*}
(see \cite[Section~3.2]{bip2021}),
where $w = \chi_T / \vert T \vert$ and $\chi_T : D \to \{ 0, 1\}$ is the characteristic function of the triangle 
$T = \text{conv}\{(1,0), (1,1), (0,1)\} \cap D \subset D$.
The initial mesh $\TT_0$ is a uniform mesh of 384 right-angled triangles.

%%%%%%%%%%%%%%%%%%%%
$\bullet$ \emph{Setup~3: second moment of a linear goal functional}.
%%%%%%%%%%%%%%%%%%%%
The physical domain is the unit square domain $D = (0,1)^2$.
\rev{The decay rate of coefficient amplitudes
in~\eqref{eq:Eigel:coefficient} is set to $\sigma=4/3$.}
Inspired by \cite[Example~7.3]{ms2009}, we choose the right-hand side function $\ff$ such that
\begin{equation*}
F(\vv)
=
- \int_\Gamma \int_{T_f} \frac{\partial \vv}{\partial x_1}(x,\y) \, \d{x} \, \d{\pi}(\y)
\quad
\text{for all } \vv \in \V,
\end{equation*}
where 
$T_f = \text{conv}\{(0,0), (1/2,0), (0,1/2)\} \cap D \subset D$.
In the spirit of~\cite[Section~3.4(b)]{tsgu2013}, the goal functional is given by the (rescaled) second moment of a linear functional.
Specifically, we consider the following goal functional:
\begin{equation*}
\gbold(\uu) = 100 \int_\Gamma \left( \int_D w(x) \, \uu(x,\y) \, \d{x} \right)^2 \d{\pi}(\y),
\end{equation*}
where $w = \chi_{T_g} / \vert T_g \vert$.
Here, $\chi_{T_g} : D \to \{ 0, 1\}$ is the characteristic function of the triangle
$T_g = \text{conv}\{(1,1/2), (1,1), (1/2,1)\} \cap D \subset D$.
The initial mesh $\TT_0$ is a uniform mesh of 512 right-angled triangles.

%%%%%%%%%%%%%%%%%%%%
$\bullet$ \emph{Setup~4: variance of a linear goal functional}.
%%%%%%%%%%%%%%%%%%%%
Let $D_\delta = (-1,1)^2 \setminus \overline{T}_\delta$, where $T_\delta = \text{conv}\{(0,0), (-1,\delta), (-1,-\delta)\}$.
In this test case, we aim at performing computations on the (physical) slit domain $(-1,1)^2 \setminus ([-1,0]\,\times\,\{0\})$.
The slit domain is not Lipschitz, however,
it is well known that an elliptic problem on this domain can be seen as the limit of the problems posed on the Lipschitz domain
$D_\delta$ as $\delta \to 0$.
Therefore, we set $D = D_\delta$ with $\delta = 0.005$.
\rev{The decay rate of coefficient amplitudes in~\eqref{eq:Eigel:coefficient} is $\sigma=2$.}
The right-hand side function is constant: $\ff \equiv 1$ in $D$.
The goal functional is given by the (rescaled) variance of a linear functional.
Specifically, we consider the following goal functional:
\begin{equation} \label{eq:variance_goal}
\gbold(\uu) = 100 \, \text{Var}_{\y} \left[ \int_D w(x) \, \uu(x,\y) \, \d{x} \right],
\end{equation}
where the weight function $w \in L^\infty(D)$ is a mollifier centered at $x_0 = (2/5 , -1/2)$ with radius $r = 3/20$
(we refer to~\cite[equation~(58)]{bprr18+} for the specific expression).
Thus, the integral over $D$ in~\eqref{eq:variance_goal} approximates the function value $\uu(x_0,\y)$
for each $\y \in \Gamma$.
The initial mesh $\TT_0$ is a uniform mesh of 512 right-angled triangles.

We note that the four nonlinear goal functionals considered in this section satisfy inequality~\eqref{eq:nonlinear:ass:goal}
with $\Cgoal > 0$ depending only on $\norm{w}{L^{\infty}(D)}$ and the Poincar\'e constant of the physical domain~$D$.

\revx{In all experiments, we run Algorithm~\ref{algorithm} with $\theta=1/2$ in~\eqref{eq:doerfler:primal}--\eqref{eq:nonlinear:doerfler:dual}.
The Monte Carlo-based validation of our goal-oriented error estimates is performed for Setups~1--3.
Note that using the bound in~\eqref{eq:truncation:error} to determine the truncation parameter $M_{\aa}$ in~\eqref{eq:truncated:coeff} results in excessively high values of $M_{\aa}$.
Indeed, the inequality $\sum_{m=M_{\aa}+1}^{\infty} \|a_m\|_{L^\infty(D)} < \sqrt{\tol}$ holds
for $M_{\aa} = 7 \cdot 10^5$ in Setup~1, for $M_{\aa} \gg 10^9$ in Setup~2, and for $M_{\aa} = 10^9$ in Setup~3.
Using these values of $M_{\aa}$ is computationally infeasible.
However, the number of parameters activated by the adaptive algorithm in Setups~1--3 does not exceed 16
(see the values $M_{\PPP_L}$ in Table~\ref{tab:results}).
Therefore, in our experiments, we choose $M_{\aa} = 100$ that should be sufficient to perform the Monte Carlo-based validation for these specific test examples.
Performing this validation with $M_{\aa} \ge 10^4$ remains an open question.
The number of Monte Carlo samples is chosen as $M_\mathrm{MC} = 150$ to ensure that the empirical variance $\text{Var}_G$
satisfies the inequality $\sqrt{\text{Var}_G / M_\mathrm{MC}} \le \tol/10$.}

%%%%%%%%%%%%%%%%%%%%
\subsection{Results}
%%%%%%%%%%%%%%%%%%%%

In Figure~\ref{fig:exp2_pictures}, for all setups,
we show the adaptively refined mesh associated with the zero index
at an intermediate step of Algorithm~\ref{algorithm}.
We observe that, in all cases, the meshes capture the spatial features of
the primal and dual solutions;
these features are induced by the geometry of the physical domain as well as by the local features of the chosen
right-hand side function $\ff$ and goal functional $\gbold$.
The intensity of local mesh refinement reflects the strength of the singularity;
e.g., in the plot for Setup~2 (top-right), the local mesh refinement at the reentrant corner of the L-shaped domain
is stronger than the one due to the local support of the weight function $w$.

%%%%%%%%%%%%%%%%%%%%
\begin{figure}[ht]
\begin{minipage}[c]{0.3\textwidth}
\centering
Setup~1 (\rev{$\TT_{14,\0}$})\\
\smallskip
\includegraphics*[height=4.5cm, trim = 117 44 97 30]{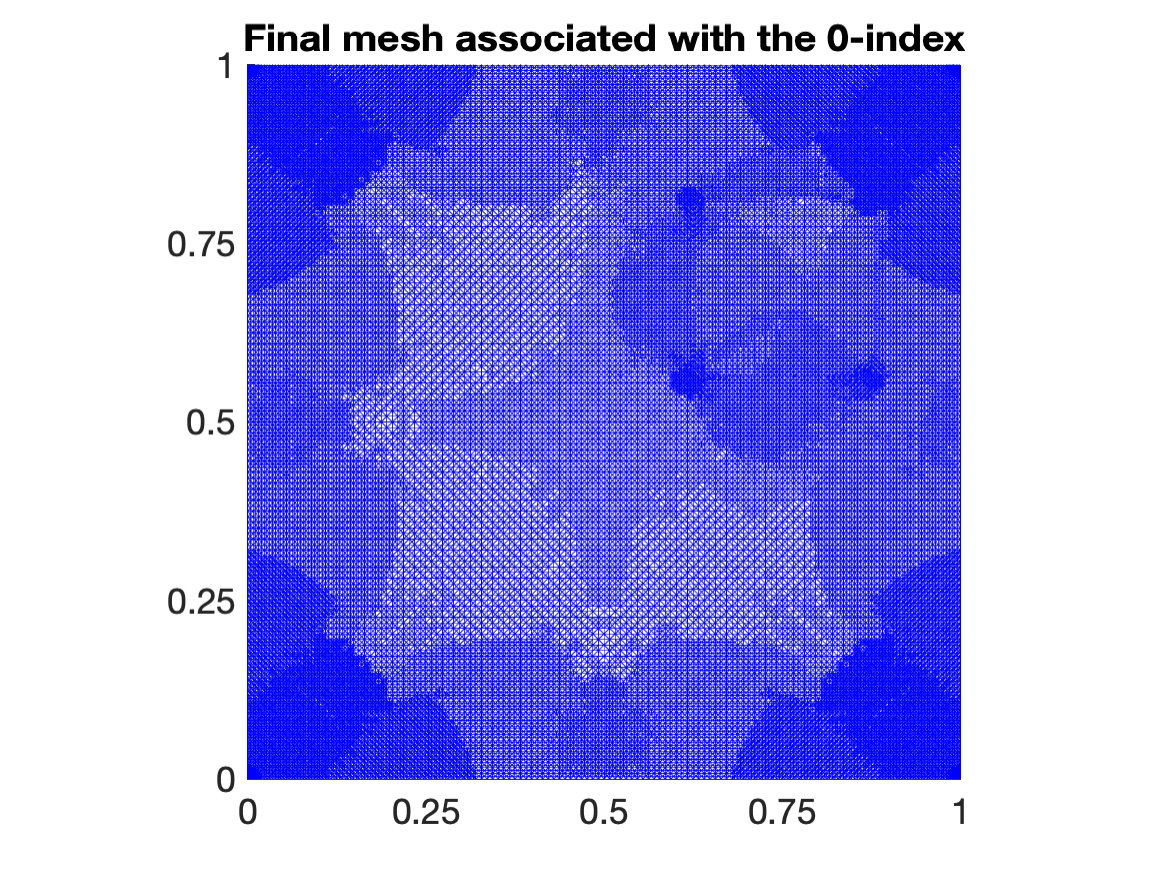}
\end{minipage}
\qquad
\begin{minipage}[c]{0.3\textwidth}
\centering
Setup~2 (\rev{$\TT_{14,\0}$})\\
\smallskip
\includegraphics*[height=4.5cm, trim = 267 87 225 59]{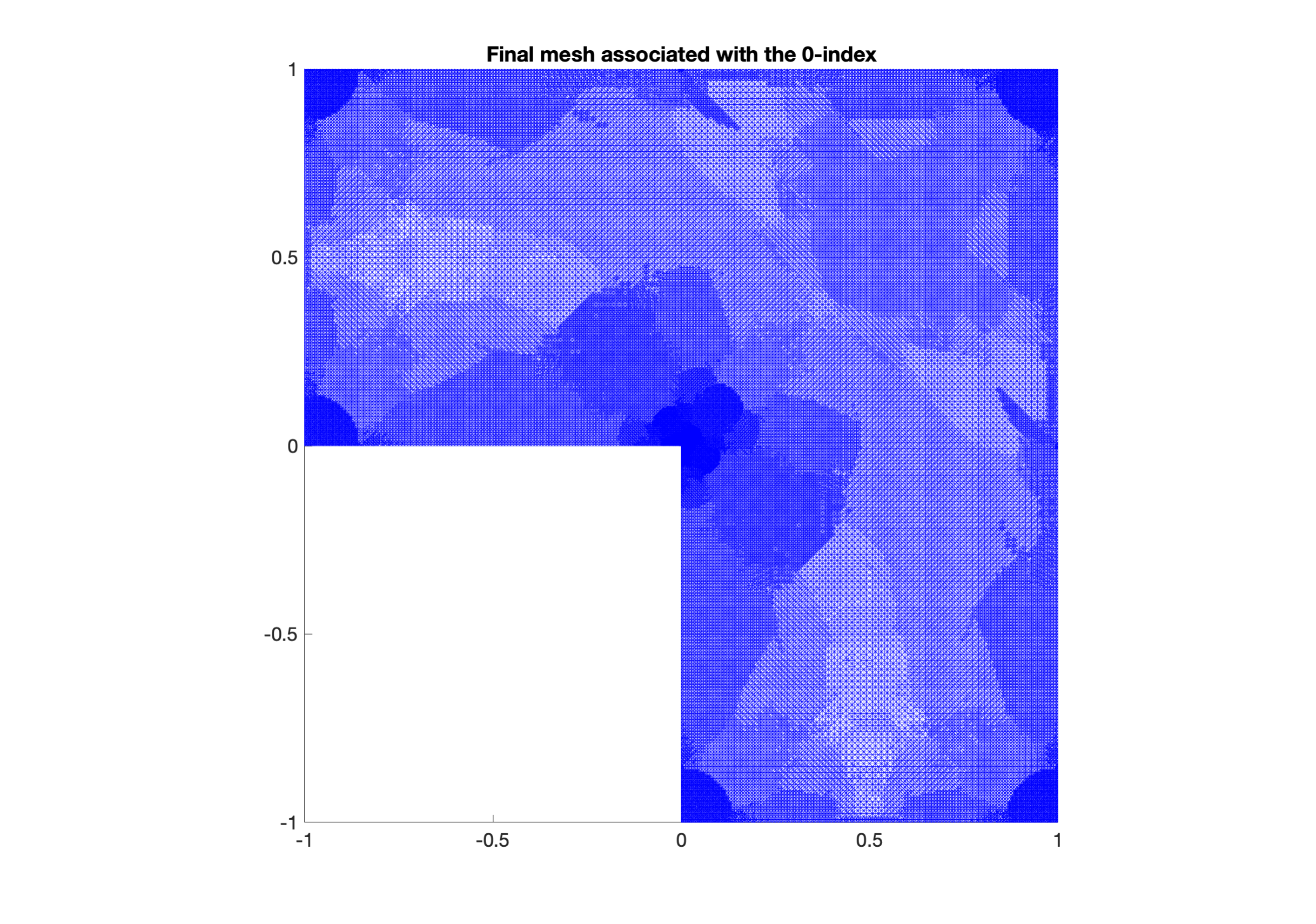}
\end{minipage}\\
\bigskip
\begin{minipage}[c]{0.3\textwidth}
\centering
Setup~3 (\rev{$\TT_{14,\0}$})\\
\smallskip
\includegraphics*[height=4.5cm, trim = 151 87 118 58]{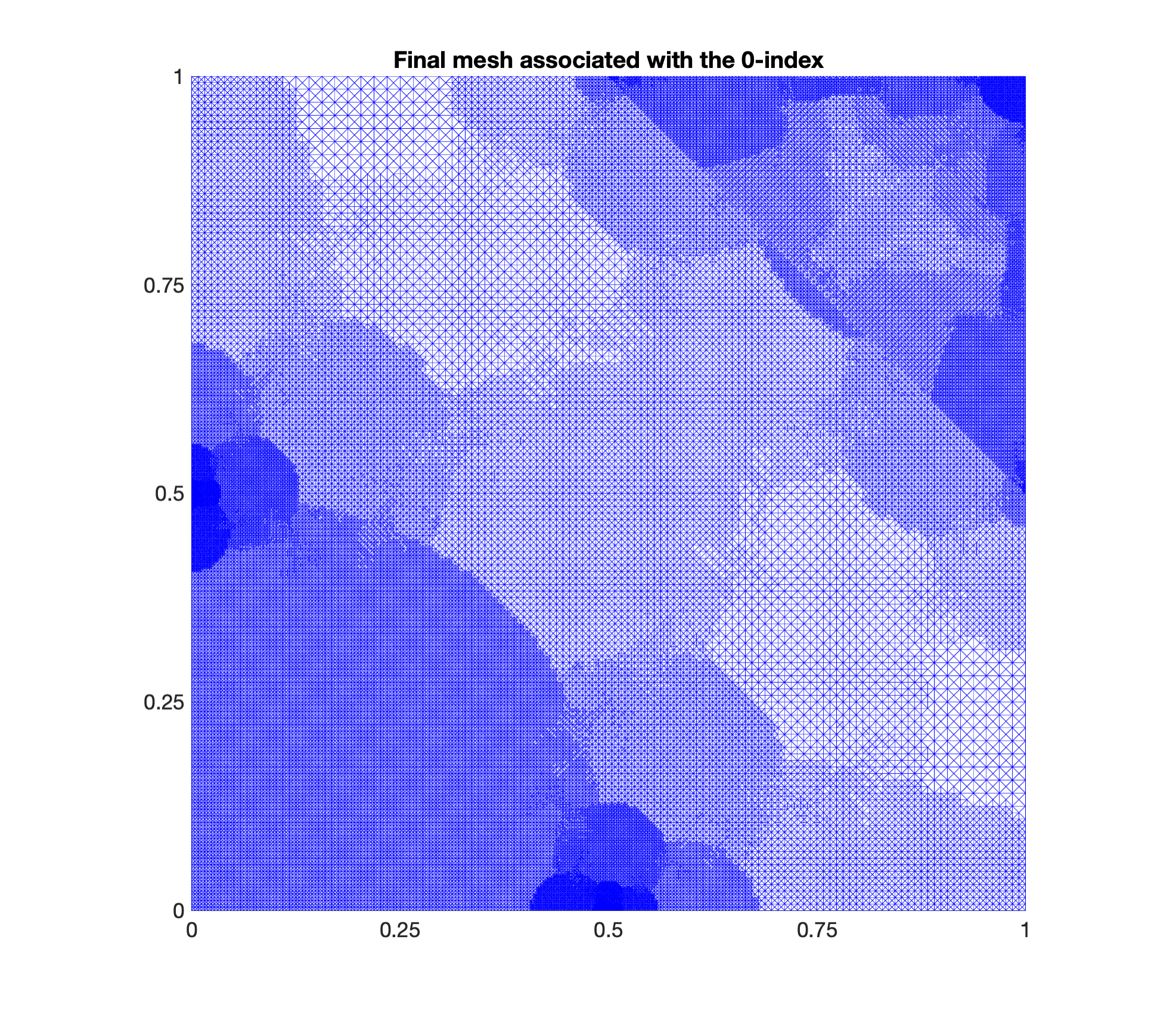}
\end{minipage}
\qquad
\begin{minipage}[c]{0.3\textwidth}
\centering
Setup~4 ($\TT_{14,\0}$)
\smallskip
\includegraphics*[height=4.5cm, trim = 117 44 97 30]{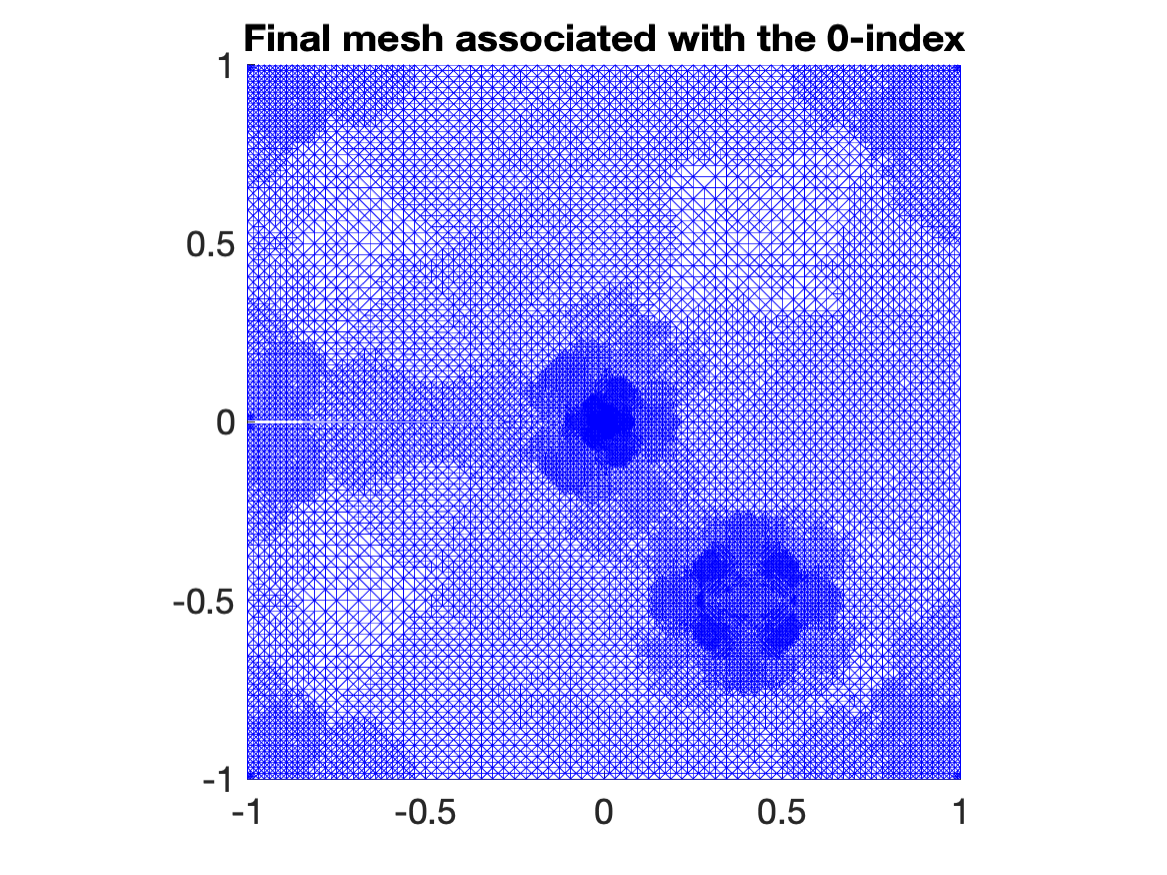}
\end{minipage}
\caption{
Adaptively refined meshes associated with the zero index
at an intermediate step of Algorithm~\ref{algorithm} for all four setups.}
\label{fig:exp2_pictures}
\end{figure}
%%%%%%%%%%%%%%%%%%%%

%%%%%%%%%%%%%%%%%%%%
\begin{figure}[ht]
\begin{tikzpicture}
\pgfplotstableread{data/setup1_sn07_decay15.dat}{\nonlinearOne}
\begin{loglogaxis}
[
title = Setup~1,
width = 7.5cm, height=6cm,						% figure's size
xlabel={number of DOFs, $N_\ell$}, 					% x label
ymajorgrids=true, xmajorgrids=true, grid style=dashed,	% grid
xmin = (1.5)*10^(2),
xmax = (3)*10^(5),
ymin = (8.0)*10^(-8),
ymax = (0.9)*10^(-3),
legend style={legend pos=south west, legend cell align=left, fill=none, draw=none}
]
\addplot[red, thick, solid, mark=o]		table[x=dofs, y=error_product]{\nonlinearOne};
\addplot[blue, thick, solid, mark=square]		table[x=dofs, y=truegerr]{\nonlinearOne};
\addplot[teal, thick, solid, mark=triangle]		table[x=dofs, y=MCerror]{\nonlinearOne};
\addplot[black,dashed,domain=10^(3.0):10^(6)] {0.8*x^(-1) };
\node at (axis cs:3e5,3e-5) [anchor=south east] {$\mathcal{O}(N_\ell^{-1})$};
\legend{
$\mu_\ell (\mu_\ell^2 + \zeta_\ell^2)^{1/2}$,
$\lvert \gbold(\uu_\mathrm{ref}) - \gbold(\uu_\ell) \rvert$,
$e_\ell^\mathrm{MC}$
}
\end{loglogaxis}
\end{tikzpicture}
%%%%%
\quad
%%%%%
\begin{tikzpicture}
\pgfplotstableread{data/setup2_sn09_decay11.dat}{\nonlinearTwo}
\begin{loglogaxis}
[
title = Setup~2,
width = 7.5cm, height=6cm,						% figure's size
xlabel={number of DOFs, $N_\ell$}, 					% x label
ymajorgrids=true, xmajorgrids=true, grid style=dashed,	% grid
xmin = (1.0)*10^(2),
xmax = (3.5)*10^(5),
ymin = (4.0)*10^(-7),
ymax = (0.6)*10^(-2),
legend style={legend pos=south west, legend cell align=left, fill=none, draw=none}
]
\addplot[red, thick, solid, mark=o]		table[x=dofs, y=error_product]{\nonlinearTwo};
\addplot[blue, thick, solid, mark=square]		table[x=dofs, y=truegerr]{\nonlinearTwo};
\addplot[teal, thick, solid, mark=triangle]		table[x=dofs, y=MCerror]{\nonlinearTwo};
\addplot[black,dashed,domain=10^(3.0):10^(6)] {7.0*x^(-1) };
\node at (axis cs:3e5,2e-4) [anchor=south east] {$\mathcal{O}(N_\ell^{-1})$};
\legend{
$\mu_\ell (\mu_\ell^2 + \zeta_\ell^2)^{1/2}$,
$\lvert \gbold(\uu_\mathrm{ref}) - \gbold(\uu_\ell) \rvert$,
$e_\ell^\mathrm{MC}$
}
\end{loglogaxis}
\end{tikzpicture}\\
%%%%%
\bigskip
%%%%%
\begin{tikzpicture}
\pgfplotstableread{data/setup3_sn08_decay43.dat}{\nonlinearThree}
\begin{loglogaxis}
[
title = Setup~3,
width = 7.5cm, height=6cm,						% figure's size
xlabel={number of DOFs, $N_\ell$}, 					% x label
ymajorgrids=true, xmajorgrids=true, grid style=dashed,	% grid
xmin = (1.5)*10^(2),
xmax = (5.0)*10^(5),
ymin = (0.3)*10^(-8),
ymax = (1.5)*10^(-3),
legend style={legend pos=south west, legend cell align=left, fill=none, draw=none}
]
\addplot[red, thick, solid, mark=o]	table[x=dofs, y=error_product]{\nonlinearThree};
\addplot[blue, thick, solid, mark=square]		table[x=dofs, y=truegerr]{\nonlinearThree};
\addplot[teal, thick, solid, mark=triangle]		table[x=dofs, y=MCerror]{\nonlinearThree};
\addplot[black,dashed,domain=10^(3.0):10^(6.0)] {2.0*x^(-1) };
\node at (axis cs:5e5,6e-5) [anchor=south east] {$\mathcal{O}(N_\ell^{-1})$};
\legend{
$\mu_\ell (\mu_\ell^2 + \zeta_\ell^2)^{1/2}$,
$\lvert \gbold(\uu_\mathrm{ref}) - \gbold(\uu_\ell) \rvert$,
$e_\ell^\mathrm{MC}$
}
\end{loglogaxis}
\end{tikzpicture}
%%%%%
\quad
%%%%%
\begin{tikzpicture}
\pgfplotstableread{data/setup4_sn10.dat}{\nonlinearFour}
\begin{loglogaxis}
[
title = Setup~4,
width = 7.5cm, height=6cm,						% figure's size
xlabel={number of DOFs, $N_\ell$}, 					% x label
ymajorgrids=true, xmajorgrids=true, grid style=dashed,	% grid
xmin = (1.5)*10^(2),
xmax = (4)*10^(5),
ymin = (5.0)*10^(-6),
ymax = (4.0)*10^(-2),
legend style={legend pos=south west, legend cell align=left, fill=none, draw=none}
]
\addplot[red, thick, solid, mark=o]		table[x=dofs, y=error_product]{\nonlinearFour};
\addplot[blue, thick, solid, mark=square]		table[x=dofs, y=truegerr]{\nonlinearFour};
\addplot[black,dashed,domain=10^(3.0):10^(6)] {100.0*x^(-1) };
\node at (axis cs:4e5,2e-3) [anchor=south east] {$\mathcal{O}(N_\ell^{-1})$};
\legend{
$\mu_\ell (\mu_\ell^2 + \zeta_\ell^2)^{1/2}$,
$\lvert \gbold(\uu_\mathrm{ref}) - \gbold(\uu_\ell) \rvert$,
}
\end{loglogaxis}
\end{tikzpicture}
%%%%%
\caption{
\rev{Evolution of the error estimates $\mu_\ell (\mu_\ell^2 + \zeta_\ell^2)^{1/2}$,
the reference errors $\lvert \gbold(\uu_\mathrm{ref}) - \gbold(\uu_\ell) \rvert$, and,
for Setups~1--3, the Monte Carlo-based approximations $e_\ell^\mathrm{MC}$
of the error in the goal functional
at each iteration of the goal-oriented adaptive algorithm.}}
\label{fig:exp2_convergence}
\end{figure}
%%%%%%%%%%%%%%%%%%%%

\rev{In Figure~\ref{fig:exp2_convergence},
for all setups,
we plot the error estimates $\mu_\ell (\mu_\ell^2 + \zeta_\ell^2)^{1/2}$ (red circular markers) and
the SGFEM-based reference errors $\lvert \gbold(\uu_\mathrm{ref}) - \gbold(\uu_\ell) \rvert$ (blue square markers); for Setups~1--3, we also plot
the Monte Carlo-based approximations $e_\ell^\mathrm{MC}$
of the error in the goal functional (green triangular markers);
all these quantities are plotted against
the number of DOFs at each iteration of the adaptive algorithm.
Looking at the plots, we} observe that, for all setups, the goal-oriented adaptive algorithm drives
the error estimate to zero, thus confirming the result of Theorem~\ref{thm:nonlinear:plain_convergence}.
Furthermore, we see that in each setup, the error estimate provides an upper bound
for the reference error
\rev{computed using the SGFEM reference solution and, in Setups~1--3,
for the approximation of the error obtained using Monte Carlo sampling.}
Finally,
all plots in Figure~\ref{fig:exp2_convergence} show that
\rev{the error estimates, the SGFEM-based reference errors,
and, in Setups~1--3, Monte Carlo-based approximations of the error in the goal functional
all decay with the rate
$N_\ell^{-1}$, which is}
the best possible decay rate achievable by conforming first-order finite elements.
Although the rate optimality property of Algorithm~\ref{algorithm} for nonlinear goal functionals
is not currently covered by our theoretical analysis, the results presented in Figure~\ref{fig:exp2_convergence}
seem to suggest that this property does hold
at least for problems with sufficiently fast decaying amplitudes of the coefficients in expansion~\eqref{eq1:a}
and for certain types of nonlinear functionals;
see \cite{bip2021} for first theoretical results in the parameter-free setting for the case of a quadratic goal~functional.

%%%%%%%%%%%%%%%%%%%%
\section{Concluding remarks} \label{sec:conclusions}
%%%%%%%%%%%%%%%%%%%%

The design of provably efficient solution strategies for high-dimensional parametric PDEs is important for reliable uncertainty quantification.
Adaptive algorithms are indispensable in this context, as they provide computationally cost-effective mechanisms for
generating accurate approximations and accelerating convergence.
In this paper, we have designed a provably convergent goal-oriented SGFEM-based adaptive algorithm
for accurate approximation of quantities of interest---linear or nonlinear
functionals of solutions to elliptic PDEs with inputs depending on infinitely many parameters.
Our theoretical results and algorithmic developments are valid for spatial domains in $\R^2$ and $\R^3$.
Our numerical results (for two-dimensional spatial domains and nonlinear goal functionals) show that
employing the \emph{multilevel} SGFEM for approximating
the primal and dual solutions leads to converging approximations of the underlying quantities of interest.
In the case of bounded linear goal functionals and for spatial domains in $\R^2$,
the rate optimality property (in the sense of approximation classes) can be proved under an appropriate saturation assumption (see~\cite{bpr25}).
The extension of this result to nonlinear goal functionals and three-dimensional domains is non-trivial and will be the subject
of future research.

%%%%%%%%%%%%%%%%%%%%

\ifsisc
%----------- SISC version -------------
%
\else
%----------- arXiv version -------------
\newpage
\appendix
%!TEX root = nonlin_goafem_ml_sgfem.tex
%%%%%%%%%%%%%%%%%%%%
\section{Proof of Lemma~\ref{lemma:lebesque}} \label{sec:appendix_a}

%%%%%%%%%%%%%%%%%%%%
%\begin{proof}
The proof is split into two steps.

{\bf Step~1.}
Let $\eps > 0$.
%Due to the assumptions of the lemma,
For arbitrary $N \in \N$ we can choose $k_0 \in \N$ such that
\begin{align*}
 \sum_{n = 1}^{N} |\alpha_n - \alpha_n^{(k)}| + C \sum_{n = 1}^\infty |\beta_n - \beta_n^{(k)}| < \eps
 \quad \text{for all } k \ge k_0.
\end{align*}
Hence, by using the triangle inequality and the fact that $|\alpha_n^{(k)}| \le C \, |\beta_n^{(k)}|$ we find that
\begin{align*}
 \sum_{n = 1}^{N} |\alpha_n| &\le 
 \sum_{n = 1}^{N} |\alpha_n^{(k)}| + \sum_{n = 1}^{N} |\alpha_n - \alpha_n^{(k)}| \le
 C \sum_{n = 1}^{N} |\beta_n^{(k)}| + \sum_{n = 1}^{N} |\alpha_n - \alpha_n^{(k)}|
 \\
 &\le
 C \sum_{n = 1}^{N} |\beta_n| + C \sum_{n = 1}^{N} |\beta_n - \beta_n^{(k)}| + \sum_{n = 1}^{N} |\alpha_n - \alpha_n^{(k)}| <
 C \sum_{n = 1}^{\infty} |\beta_n| + \eps
\end{align*}
for any $\eps > 0$ and for arbitrary $N \in \N$. Therefore,
$%\begin{align*}
 \sum_{n = 1}^{\infty} |\alpha_n| \le C\,\sum_{n = 1}^{\infty} |\beta_n| < \infty.
$%\end{align*}

{\bf Step~2.}
Let $\eps\,{>}\,0$.
Since $\sum_{n = 1}^\infty |\alpha_n| + C \sum_{n = 1}^\infty |\beta_n| < \infty$,
we can choose $n_0 \in \N$ such~that 
\begin{align*}
 \sum_{n = n_0}^\infty |\alpha_n| + C \sum_{n = n_0}^\infty |\beta_n| < \eps.
\end{align*}
Due to the assumed convergence, we can choose $k_0 \in \N$ such that
\begin{align*}
 \sum_{n = 1}^{n_0-1} |\alpha_n - \alpha_n^{(k)}| + C \sum_{n = 1}^\infty |\beta_n - \beta_n^{(k)}| < \eps
 \quad \text{for all } k \ge k_0.
\end{align*}
Then, for all $k \ge k_0$, the triangle inequality yields that
\begin{align*}
 \sum_{n = 1}^\infty |\alpha_n - \alpha_n^{(k)}|
 &\le \sum_{n = 1}^{n_0-1} |\alpha_n - \alpha_n^{(k)}| + \sum_{n = n_0}^\infty |\alpha_n| + \sum_{n = n_0}^\infty |\alpha_n^{(k)}|
 \\
 &\le \sum_{n = 1}^{n_0-1} |\alpha_n - \alpha_n^{(k)}| + \sum_{n = n_0}^\infty |\alpha_n| + C \sum_{n = n_0}^\infty |\beta_n^{(k)}|
 \\
 &\le \sum_{n = 1}^{n_0-1} |\alpha_n - \alpha_n^{(k)}| + \sum_{n = n_0}^\infty |\alpha_n| +
   C \sum_{n = n_0}^\infty |\beta_n| + C \sum_{n = n_0}^\infty |\beta_n - \beta_n^{(k)}|
 < 2 \eps. 
\end{align*}
This concludes the proof.
%\end{proof}
%%%%%%%%%%%%%%%%%%%%

\fi

%%%%%%%%%%%%%%%%%%%%
\bibliographystyle{alpha}
\bibliography{literature}
%%%%%%%%%%%%%%%%%%%%

%%%%%%%%%%%%%%%%%%%%
\end{document}